\documentclass[12pt]{amsart}

\usepackage{latexsym, fancyhdr, enumitem}
\usepackage{amsmath, amsfonts, amssymb}
\usepackage[english]{babel}
\usepackage{graphicx}

\newtheorem{prop}{Proposition}[section]
\newtheorem{lem}[prop]{Lemma}
\newtheorem{fol}[prop]{Corollary}
\newtheorem{theo}[prop]{Theorem}

\theoremstyle{definition}
\newtheorem{defi}[prop]{Definition}

\theoremstyle{remark}
\newtheorem{rem}[prop]{Remark}
\newtheorem{remdef}[prop]{Definition and Remark}
\newtheorem{exmp}[prop]{Example}

\newcommand{\field}[1]{\mathbb{#1}}

\newcommand{\ZZ}{\field{Z}}

\newcommand{\RR}{\field{R}}
\newcommand{\NN}{\field{N}}

\newcommand{\KKK}{\mathcal{K}}
\newcommand{\PP}{\mathcal{P}}

\newcommand{\RRR}{\field{R}}
\newcommand{\SSS}{\field{S}}
\newcommand{\SSSS}{\mathcal{S}}
\newcommand{\AAA}{\mathcal{A}}

\DeclareMathOperator*{\codim}{\mathrm{codim}}
\DeclareMathOperator*{\rk}{\mathrm{rk}}
\DeclareMathOperator*{\corank}{\mathrm{corank}}

\DeclareMathOperator{\id}{\mathrm{id}}

\DeclareMathOperator{\supp}{\mathrm{supp}}

\DeclareMathOperator{\proj}{\mathrm{proj}}
\DeclareMathOperator*{\GL}{\mathrm{GL}}

\DeclareMathOperator*{\Rad}{\mathrm{Rad}}

\DeclareMathOperator*{\Cham}{\mathrm{Cham}}
\DeclareMathOperator*{\red}{\mathrm{red}}
\DeclareMathOperator*{\St}{\mathrm{St}}

\newcommand{\ol}[1]{\overline{#1}}
\newcommand{\ul}[1]{\underline{#1}}

\title[Simplicial arrangements on convex cones]
{Simplicial arrangements on convex cones}

\author{M.~Cuntz}
\address{Michael Cuntz,
Institut f\"ur Algebra, Zahlentheorie und Diskrete Mathematik,
Fakult\"at f\"ur Mathematik und Physik,
Leibniz Universit\"at Hannover,
Welfengarten 1,
D-30167 Hannover, Germany}
\email{cuntz@math.uni-hannover.de}

\author{B.~M\"uhlherr}
\address{Bernhard M\"uhlherr,
Mathematisches Institut, Arndtstra{\ss}e 2, 35392 Gie{\ss}en, Germany}
\email{bernhard.muehlherr@math.uni-giessen.de}

\author{C.~J.~Weigel}
\address{Christian J. Weigel,
Mathematisches Institut, Arndtstra{\ss}e 2, 35392 Gie{\ss}en, Germany}
\email{christian.j.weigel@math.uni-giessen.de}

\begin{document}

\begin{abstract}
We introduce the notion of a Tits arrangement on a convex open cone as a special case of (infinite) simplicial arrangements. Such an object carries a simplicial structure similar to the geometric representation of Coxeter groups. The standard constructions of subarrangements and restrictions, which are known in the case of finite hyperplane arrangements, work as well in this more general setting.
\end{abstract}

\maketitle

\tableofcontents

\section*{Introduction}

Let $(W,S)$ be a Coxeter system. There is a canonical real representation
of $W$ which is called the {\sl geometric representation} of $(W,S)$ and which
is known to be faithful by the work of Tits in the 1960s (a proof can be found in \cite{Ti13}.
In the proof he considers a convex cone $T$ in the dual module
which is stabilized under the action of $W$. By means of the standard generating set $S$
he defines an open simplicial cone $C$ which turns out to be a prefundamental domain for
the action of $W$ on $T$. 

It is now the standard terminology to call the cone $T$  the {\sl Tits cone}
and $C$ the {\sl (open) fundamental chamber}. The conjugates of the 
fundamental generators act as linear reflections on the module and its dual. Their reflection hyperplanes
yield a simplicial decomposition of $T$ and thus provide a simplicial complex
which is the {\it Coxeter complex} of $(W,S)$. The geometric representation of $(W,S)$,
the geometry of the Tits cone, and the properties of the Coxeter complex are fundamental
tools for the investigation of Coxeter groups.

Coxeter systems play an important role in various branches of mathematics. In combinatorics
and geometric group theory they are a rich source of interesting phenomena (\cite{BB05}, \cite{Da08}). In 
representation theory and the theory of algebraic groups the
{\sl crystallographic} Coxeter systems occur as {\sl Weyl groups} of several
structures. There, they play the role of basic invariants, often called the {\sl type} of
the algebraic object under consideration.

In \cite{HY08} and \cite{p-CH09a} {\sl Weyl groupoids} and their 
root systems have been introduced as basic invariants of Nichols algebras.
Their definition emerged from earlier work in \cite{p-H-06} and \cite{p-AHS-08}, they are
natural generalizations of Weyl groups. 
Weyl groupoids of finite type are those which generalize finite Weyl groups.
These have been studied intensively in \cite{CH09}, \cite{p-CH10}, \cite{CH12}, \cite{CH11}, where
a classification has been obtained. This classification is considerably harder 
than the classification of the finite Weyl groups. It was observed by the first
two authors in 2009 that the final outcome of this classification 
in rank 3 is intimately
related to Gr\"unbaum's list of simplicial arrangements in \cite{p-G-71}. There is an obvious
explanation of this connection: To each Weyl groupoid one can associate a Tits cone
which is the whole space if and only if the Weyl groupoid is of finite type.
Based on these considerations, {\sl crystallographic} arrangements have been introduced
in \cite{Cu11}, where it is shown that their classification is a consequence of the classification of finite
Weyl groupoids.

In view of the importance of the Tits cone and the Coxeter complex for 
the understanding of Coxeter systems it is natural to investigate their analogues
in the context of Weyl groupoids. Although it is intuitively clear what has to 
be done in order to generalize these notions, there are
instances where things have to be modified or some extra argument is needed.
Our intention is to provide the basic theory of the Tits cone and the Coxeter complex
of a Weyl groupoid. This paper deals with the combinatorial aspects of this project and therefore
the crystallographic condition doesn't play a role at all. Hence, here we deal with
{\sl Tits arrangements} rather than with Weyl groupoids.

The basic strategy for setting
up the framework is to start with a simplicial arrangement ${\AAA}$ on a convex open cone $T$ and
to investigate the abstract simplicial complex ${\mathcal{S}}({\AAA},T)$ associated with it. This is a {\sl gated
chamber complex} and has therefore a natural type function. 
We call a simplicial arrangement on a convex open cone a Tits arrangement
if  ${\mathcal{S}}({\mathcal{A}},T)$ is a thin chamber complex and introduce the notion of a root system of a Tits arrangement. Given a simplicial arrangement of rank $r$
on an open convex cone, there are two canonical procedures to produce
simplicial arrangements of smaller rank.  

We would like to point out that this paper is meant to be a contribution to the foundations
of the theory of Weyl groupoids of arbitrary type. The concepts and ideas are at least
folklore and several of the results for which we give proofs are well established
in the literature. The only exception is probably our systematic use of gated
chamber complexes at some places. 
Our principal goal here is to provide a fairly complete
account of the basic theory of the Tits cone of a Weyl groupoid
by developing the corresponding notions to the extent that is needed for just this purpose.
Therefore, we have decided to include short proofs for standard facts and refer
to other sources only when we need a more elaborate result. 

We are not able to give a systematic account of the origins
of the concepts and ideas which play a role in this paper. 
Here are some comments based on the best of our knowledge:

\begin{enumerate}[label=\arabic*.]
        \item  Simplicial arrangements were first introduced and studied by Melchior \cite{a-Melchi41} and subsequently by Gr\"unbaum \cite{p-G-71}. Shortly afterwards, simplicial arrangements attracted attention in the seminal work of Deligne \cite{MR0422673}: they are a natural context to study the $K(\pi,1)$ property of complements of reflection arrangements, since the set of reflection hyperplanes of a finite Coxeter group is a simplicial arrangement.
They further appeared as examples or counterexamples to conjectures on arrangements.
\item We do not know where arrangements of hyperplanes on convex cones were considered for the
first time. The concept seems most natural and they are mentioned in 
\cite{Par14} without further reference. 
Of course, our definition of a simplicial arrangement on an open convex cone is 
inspired by the Tits cone of a Coxeter system.  
\item The fact that
arrangements of hyperplanes provide interesting examples of gated sets in metric spaces
appears in \cite{BLSWZ} for the first time. At least in the simplicial case it was  
observed much earlier \cite{Ti74}.
\item The observation that there is a natural link between root systems and simplicial arrangements
is quite natural. We already mentioned that it was our starting point to investigate
the Tits-cone of a Weyl groupoid. But it also appears in Dyer's work on rootoids \cite{Dye11}, \cite{Dye11-2}.
It is conceivable that the observation was made much earlier by other people and is
hidden somewhere in the literature.
\end{enumerate}

The paper is organized as follows.

In Section \ref{sec:ex} we will provide some basic examples for our objects of interest. These will be used for later reference, since they are either basic examples or counterexamples for the properties which we introduce. Like in Example \ref{exm:thinf}, the geometric representation of a Coxeter group is a prominent sample for simplicial arrangements.

In Section \ref{sec:hyp} we fix notation and develop the notion of a hyperplane arrangement on an open convex cone in a real vector space $V$. We introduce the common substructures for these objects, i.\ e.\ subarrangements and restrictions. Furthermore we define the chamber graph of an arrangement and show that parabolic subsets of the set of chambers are gated subsets with respect to the canonical metric.

In Section \ref{sec:tits} we introduce simplicial arrangements and Tits arrangements. We add additional combinatorial structure to a Tits arrangement by associating to the set of hyperplanes a set of roots, linear forms which define the hyperplanes. With respect to the study of Nichols algebras root systems with additional properties will be interesting, however in this paper we deal with roots systems as very general objects. We also associate to a simplicial arrangement a canonical simplicial complex. The main results regarding this complex are proven in the appendix.

In Section \ref{subarr} we consider subarrangements and restrictions of simplicial arrangements and Tits arrangements. We give criteria when the substructures of a simplicial/Tits arrangement is again a simplicial/Tits arrangement, and in the case of a Tits arrangement we describe canonical root systems.

The appendix provides proofs for the properties of the poset associated to a simplicial hyperplane arrangement which are stated in Section 4. While most of these properties are quite intuitive, a rigid proof can be tedious.

Appendix \ref{APP:simpcomp} recalls the basic definitions of simplicial complexes as we need them.

In Appendix \ref{APP:S} we show that the poset $\SSSS(\AAA, T)$ associated to a simplicial hyperplane arrangement $(\AAA, T)$ is indeed a simplicial complex. The results of this Section are summarized in Proposition \ref{Ssimpcomp}. In particular, we provide an equivalent definition for simplicial cones, which has an implicit simplicial structure.

Appendix \ref{APP:SCC} provides the remaining properties of $\SSSS$. The first part recalls the definitions of chamber complexes and type functions.

In the second part it is shown that $\SSSS$ is a gated chamber complex with a type function, a collection of the results can be found in Proposition \ref{prop-Sproperties}. We also show that the notions of being spheric and thin, which we introduced for simplicial hyperplane arrangements before, coincide with the classical notions for chamber complexes.

\medskip

\noindent \textbf{Acknowledgements.}
Part of the results were achieved during a Mini-Workshop on Nichols algebras and Weyl groupoids at the Mathematisches Forschungsinstitut Oberwolfach in October 2012, and during meetings in Kaiserslautern, Hannover, and Giessen supported by the Deutsche Forschungsgemeinschaft within the priority programme 1388. One of the authors was partly funded by a scholarship of the Justus-Liebig-Universit\"at Giessen.
\section{Introductory examples}
\label{sec:ex}

\begin{exmp}
Consider the following setting. Let 
\begin{align*}
V &= \RR^2,\\
T &= \{(x,y) \mid x,y \in \RR_{>0}\},\\
L &= \{(x,x) \mid x \in \RR\}.
\end{align*}

Then $T \setminus L$ consists of two connected components
$$K_1 := \{(x,y) \in T \mid x <y\}, K_2 := \{(x,y) \in T \mid x>y\}.$$

Let $\alpha_1$, $\alpha_2$ be the dual basis in $V^\ast$ to $(1,0), (0,1)$, then we can write
\begin{align*}
K_1 &= \alpha_1^{-1}(\RR_{>0}) \cap (-\alpha_1+ \alpha_2)^{-1}(\RR_{>0}),\\
K_2 &= \alpha_2^{-1}(\RR_{>0}) \cap (\alpha_1 - \alpha_2)^{-1}(\RR_{>0}).
\end{align*}
We will later define objects which can be written in this way, i.\ e.\ as intersections of half spaces, as simplicial cones.

Both $K_1$ and $K_2$ are cones with exactly two bounding hyperplanes, for $K_1$ these are $L = \ker(\alpha_1 - \alpha_2)$ and $\{(0,y) \mid y \in \RR\} = \ker(\alpha_1)$, for $K_2$ these are $L$ and $\{(x,0) \mid x \in \RR\} = \ker(\alpha_2)$.

Note that in this example only one of the bounding hyperplanes, namely $L$, of $K_1$, $K_2$ meets $T$, while the other meets $\ol{T}$, but not $T$ itself.
\begin{figure}[ht]%
\includegraphics[width=0.6\columnwidth]{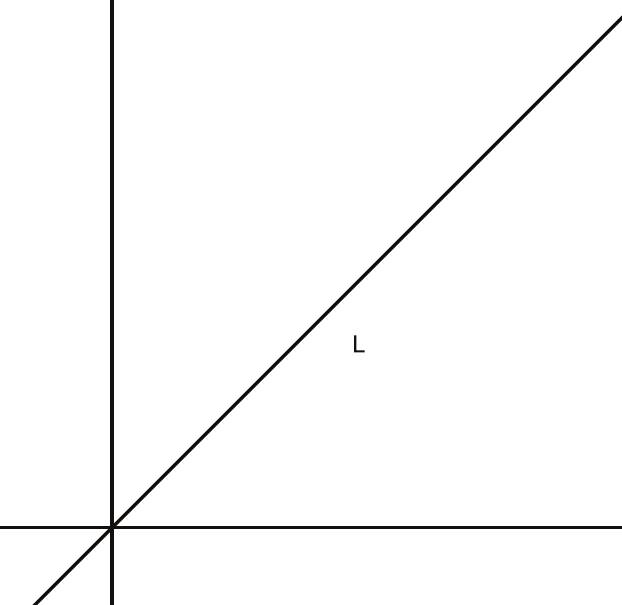}%
\caption{The line setting from Example \ref{exm:notthin}.}%
\label{fig:notthin}%
\end{figure}

\label{exm:notthin}
\end{exmp}

\begin{exmp}
Let $V$ and $T$ be as before. Define for $n \in \NN_{>0}$ the lines
\begin{align*}
L_n &= \{(x,y) \mid y - nx = 0\},\\
{L_n}' &=  \{(x,y) \mid y - \frac{1}{n}x = 0\}.
\end{align*}

Then the connected components of $T \setminus \bigcup_{n \in \NN_{>0}} (L_n \cup {L_n}')$ are again simplicial cones for suitable linear forms. The number of connected components is not finite in this case, and every component has bounding hyperplanes which meet $T$ (see Figure \ref{fig:thinf}).

\begin{figure}[ht]%
\includegraphics[width=0.6\columnwidth]{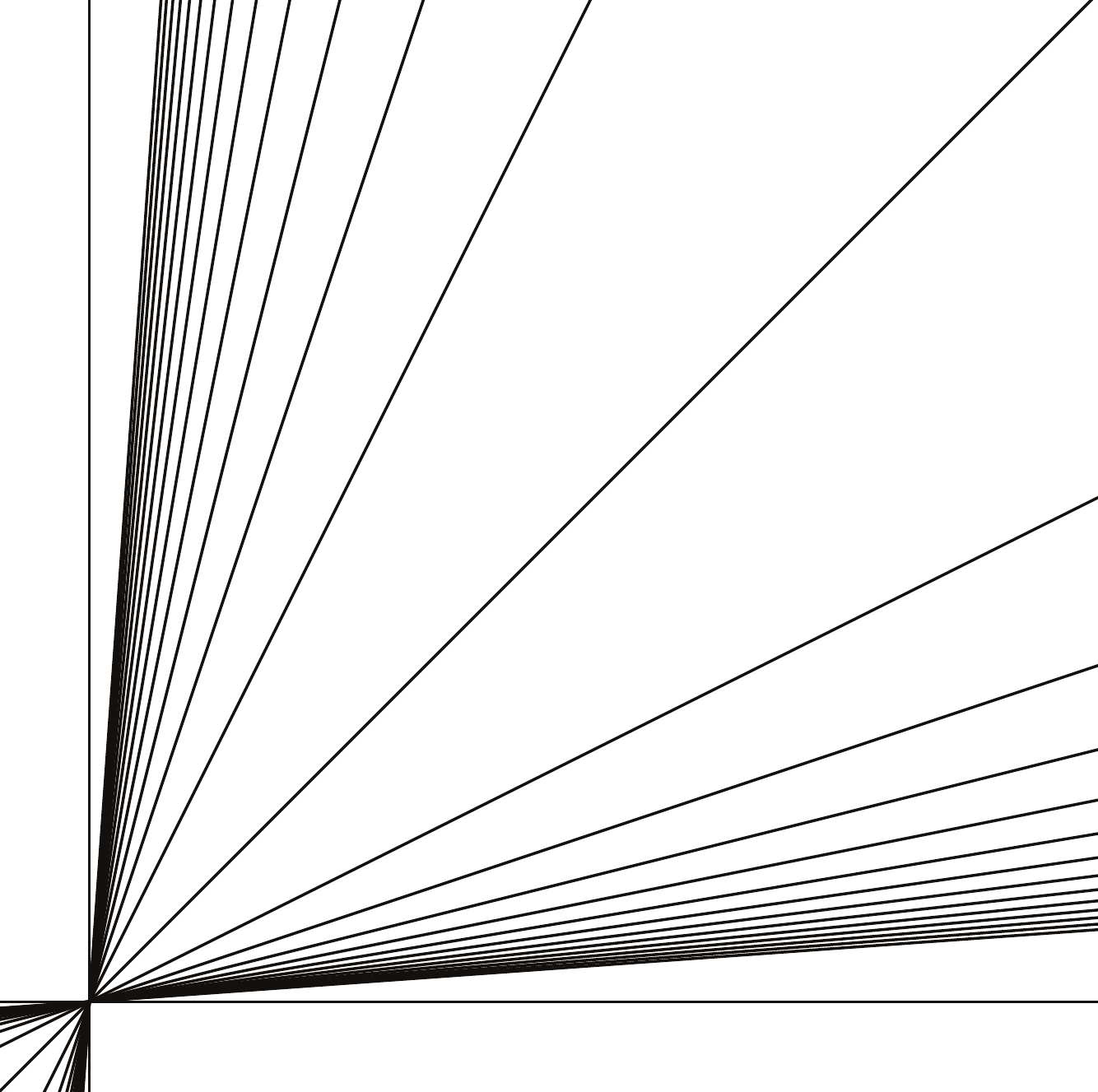}%
\caption{An infinite collection of hyperplanes, tessellating $\RR_{>0} \times \RR_{>0}$.}%
\label{fig:thinf}%
\end{figure}

\label{exm:thinf}
\end{exmp}

\begin{exmp}
Let $V = \RR^3$ and
$$T := \{(x_1, x_2, x_3) \in \RR^3 \mid x_1^2 + x_2^2 - x_3^2 < 0\}.$$

Then $T$ is a convex open cone. Consider the universal Coxeter group $W$ in 3 generators, i.\ e.
$$
W = \langle s,t,u \mid s^2 = t^2 = u^2 = 1 \rangle.
$$

The geometric representation (see \cite[Chapter 5.3]{Hu90}) of $W$ yields a set of reflection hyperplanes $\AAA$, which meet $T$ after choosing a suitable basis.

We obtain the picture in Figure \ref{fig:univ3} by intersecting $T$ with a hyperplane parallel to $\langle e_1, e_2 \rangle$, which corresponds to the Beltrami-Klein model of hyperbolic 2-space.

\begin{figure}[ht]%
\includegraphics[width=0.6\columnwidth]{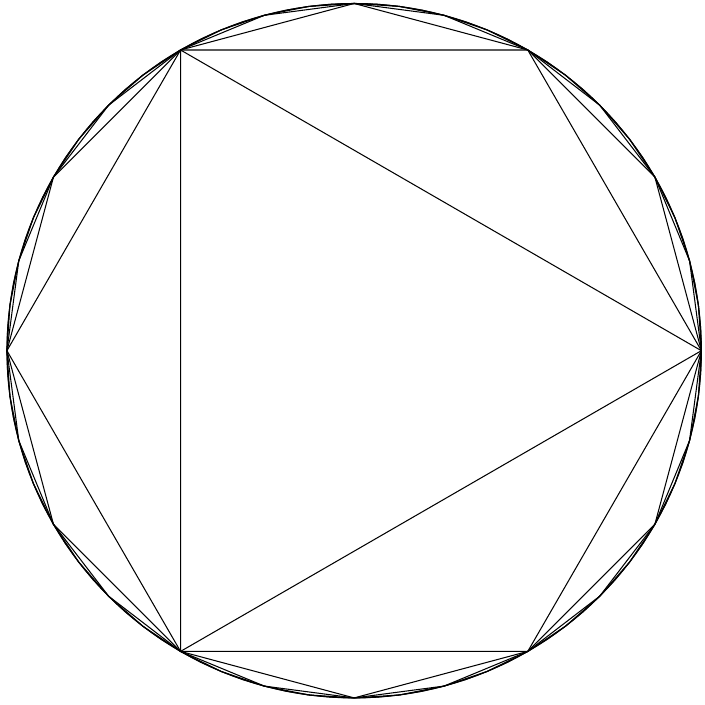}%
\caption{The hyperbolic hyperplane arrangement associated to $W$.}%
\label{fig:univ3}%
\end{figure}

The connected components of $T \setminus \bigcup_{H \in \AAA} H$ are again cones which carry a simplicial structure. However, the vertices of the simplices are not contained in $T$, but in its boundary $\partial T$.
\label{exm:univ3}
\end{exmp}

\section{Hyperplane arrangements, subarrangements and restrictions}
\label{sec:hyp}
\subsection{Hyperplane arrangements}

Throughout this paper, all topological properties are with respect to the standard topology of $\RR^r$, unless stated otherwise.

\begin{defi}
Let $V = \RR^r$. A subset $K \subset V$ is called a \textit{cone}, if $\lambda v \in K$ for all $v \in T$, $0 < \lambda \in \RR$. For a subset $X \subset V$ we call
$$\RR_{>0}X := \{ \lambda x \mid x \in X, \lambda \in \RR_{>0}\}$$
the \textit{cone over $X$}.

Let $\AAA$ be a set of linear hyperplanes in $V = \RR^r$, and $T$ an open convex cone. We say that \textit{$\AAA$ is locally finite in $T$}, if for every $x \in T$ there exists a neighborhood $U_x \subset T$ of $x$, such that $\{H \in \AAA \mid H \cap X \neq \emptyset\}$ is a finite set.

A \textit{hyperplane arrangement (of rank $r$)} is a pair $(\AAA, T)$, where $T$ is a convex open cone in $V$, and $\AAA$ is a (possibly infinite) set of linear hyperplanes such that 
\begin{enumerate}
	\item $H \cap T \neq \emptyset$ for all $H \in \AAA$,
	\item $\AAA$ is locally finite in $T$.
\end{enumerate}
If $T$ is unambiguous from the context, we also call the set $\AAA$ a hyperplane arrangement.

Let $X \subset \ol{T}$. Then the \textit{support of $X$} is defined as
$$\supp_\AAA(X) = \{H \in \AAA \mid X \subset H\}.$$
If $X= \{x\}$ is a singleton, we write $\supp_\AAA(x)$ instead of $\supp_\AAA(\{x\})$, and we omit the index $\AAA$, if $\AAA$ is unambiguous from the context. In this paper we call the set $$\sec_\AAA(X) := \bigcup_{x \in X} \supp_\AAA(x) = \{H \in \AAA \mid H \cap X \neq \emptyset\}$$ the \textit{section of $X$ (in $\AAA$)}. Again, we will omit $\AAA$ when there is no danger of confusion.

The connected components of $T \setminus \bigcup_{H \in \AAA} H$ are called \textit{chambers}, denoted with $\mathcal{K}(\AAA,T)$ or just $\mathcal{K}$, if $(\AAA,T)$ is unambiguous.

Let $K \in \KKK(\AAA,T)$. Define the \textit{walls of $K$} as 
$$W^K := \{H \leq V \mid H \text{ hyperplane, } \langle H \cap \overline{K} \rangle = H, H \cap K^\circ = \emptyset\}.$$

Define the \textit{radical of $\AAA$} as $\Rad(\AAA) := \bigcap_{H \in \AAA} H$. We call the arrangement \textit{non-degenerate}, if $\Rad(\AAA) = 0$, and \textit{degenerate} otherwise. A hyperplane arrangement is \textit{thin}, if $W^K \subset \AAA$ for all $K \in \KKK$.
\end{defi}

\begin{rem}
\begin{enumerate}
	\item By construction, the chambers $\KKK$ are open sets.
	\item In our notation, if $(\AAA,T)$ is a hyperplane arrangement, $\AAA$ being locally finite in $T$ implies: For every point $x \in T$ there exists a neighborhood $U_x \subset T$ of $x$ such that $\sec_\AAA(U_x)$ is finite.
\end{enumerate}
\end{rem}

\begin{lem}
Let $(\AAA, T)$ be a hyperplane arrangement. Then for every point $x \in T$ there exists a neighborhood $U_x$ such that $\supp(x) = \sec(U_x)$.
Furthermore the set $\sec(X)$ is finite for every compact set $X \subset T$.
\label{compfin}
\end{lem}
\begin{proof}
Let $x \in T$, in particular there exists an open neighborhood $U$ of $x$ such that $\sec(U)$ is finite. By taking the smallest open $\varepsilon$-ball contained in $U$ and centred at $x$, we can assume $U = U_\varepsilon(x)$. Let $H \in \sec(U)$, with $x \notin H$. Let $\delta = d(H,x) > 0$, then $\delta < \varepsilon$, and $U' := U_{\frac{\delta}{2}}(x)$ is an open subset such that $\sec(U') \subseteq \sec(U)\setminus \{H\}$. Since $\sec(U)$ is finite, $\sec(U) \setminus \sec(x)$ is finite. We can therefore repeat this process finitely many times until we find an open ball $B$ such that $\sec(B) = \supp(x)$.

The second assertion is a consequence of the first: Let $X$ be compact, and for $x \in X$ let $U_x$ denote an open subset such that $\sec(U_x) = \supp(x)$.
Then
$$\sec(X) \subseteq \bigcup_{x \in X} \sec(U_x)
$$ as
$X \subseteq \bigcup_{x \in X} U_x$. But the $U_x$ are open and $X$ is compact, therefore there exists a finite set $\{x_1, \dots, x_n\} \subset X$, $n \in \NN$, such that
$$
X \subseteq \bigcup_{i=1}^n U_{x_i}.
$$
Consequently, if $H\cap X \neq \emptyset$, then there exists an index $i$ such that $H \cap U_{x_i} \neq \emptyset$ and $H \in \sec(U_{x_i})$. Hence
$$\sec(X) \subseteq \bigcup_{i=1}^n \sec(U_{x_i})
$$
and $\sec(X)$ is finite.
\end{proof}

\subsection{Subarrangements}

We want to say more about hyperplane arrangements arising as $\supp(x)$ for points $x \in \ol{T}$.

\begin{defi}
Let $(\AAA,R)$ be a hyperplane arrangement and let $x \in \ol{T}$. If $\AAA' \subset \AAA$, we call $(\AAA', T)$ a \textit{subarrangement of $(\AAA,T)$}.

Define $\AAA_x := \supp(x)$, we call $\AAA_x$ the \textit{induced arrangement at $x$} or the \textit{parabolic subarrangement at $x$}. A \textit{parabolic subarrangement of $(\AAA,T)$} is a subarrangement $(\AAA',T)$ of $(\AAA,T)$, such that $\AAA' = \AAA_x$ for some $x \in \ol{T}$.

Furthermore set $\KKK_x := \{K \in \KKK\ |\ x \in \overline{K}\}$.
\end{defi}

\begin{defi}
Let $(V,d)$ be a connected metric space. Define the segment between $x,y \in V$ to be 
$$
\sigma(x,y) := \{z \in V \mid d(x,z) + d(z,y) = d(x,y)\}.
$$

When referring to the metric of $\RR^r$, we will use the more common notion of the \textit{interval} between two points $x$ and $y$:
\begin{align*}
[x,y] &:= \sigma(x,y) = \{\lambda x + (1 - \lambda) y \in \RR^r \mid \lambda \in [0,1] \},\\
(x,y) &:= \sigma(x,y)\setminus \{x,y\} = \{\lambda x + (1 - \lambda) y \in \RR^r \mid \lambda \in (0,1) \}.
\end{align*}
Note that the intervals $(x,y]$, $[x,y)$ can be defined analogously.
\end{defi}

\begin{remdef}
Note that every hyperplane $H$ separates $V$ into half-spaces. One way to describe the half-spaces uses linear forms.

Choose an arbitrary linear form $\alpha \in V^\ast$ such that $\alpha^\perp = H$. Then $\alpha^+$ and $\alpha^-$ ($\ol{\alpha^+}$ and $\ol{\alpha^-}$) are the two \textit{open (closed) half-spaces bounded by $H$}.

For an arbitrary subset $X \subset T$, if $X$ is contained in one open half space bounded by $H$, we denote this particular half-space by $D_H(X)$. In this case we write $-D_H(X)$ for the unique half-space not containing $X$. By definition every chamber $K$ is contained in a unique half space of $H \in \AAA$, therefore the sets $D_H(K)$ exist for all $H \in \AAA$, $K \in K$. Let $K,L$ be chambers, we say that $H \in \AAA$ \textit{separates} $K$ and $L$, if $D_H(K) = -D_H(L)$. We also say that two closed chambers $\ol{K},\ol{L} \in \Cham(\SSSS)$ are \textit{separated} by $H \in \AAA$, if $H$ separates $K$ and $L$.

We denote by $S(K,L) := \{H \in \AAA \mid D_H(K) \neq D_H(L)\}$ the set of hyperplanes separating $K,L$.
\end{remdef}

\begin{lem}
If $K, L \in \KKK_x$, $H \in S(K,L)$, then $H \in \AAA_x$.
\label{starsep}
\end{lem}
\begin{proof}
Assume $H \notin \AAA_x$, then the half-space $D_H(\{x\})$ is well defined and unique. As a consequence we have $D_{H}(K) = D_{H}(\{x\}) = D_{H}(L)$, and $H$ does not separate $K$ and $L$.
\end{proof}

\begin{lem}
The pair $(\AAA_x, T)$ is a hyperplane arrangement with chambers corresponding to $\KKK_x$.
\label{lem:subarr}
\end{lem}
\begin{proof}
The set $\AAA_x$ is locally finite in $T$ since $\AAA$ is locally finite in $T$, and every $H \in \AAA_x$ meets $T$. Let $\KKK'$ denote the connected components of $T \setminus \bigcup_{H \in \AAA_x} H$. 

Assume $K_1, K_2 \in \KKK_x$ are both contained in $K \in \KKK'$. Then there exists a hyperplane $H \in \AAA$ such that $K_1$, $K_2$ are contained in different half-spaces with respect to $H$, in contradiction to Lemma \ref{starsep}. Since $\AAA_x \subset \AAA$, every chamber in $\KKK_x$ is therefore contained in a unique chamber in $\KKK'$.

Likewise, every chamber in $K'$ contains a chamber in $K_x$, which completes the proof.
\end{proof}

\subsection{Reductions}

We note that it is always possible to mod out the radical of an arrangement to obtain a non-degenerate one.

\begin{defi}
Assume that $(\AAA,T)$ is a hyperplane arrangement, set $W = \Rad(\AAA)$ and
\begin{enumerate}[label=\roman*)]
  \item $V^{\red} := V / W$,
	\item $\pi: V \mapsto V^{\red}, v \to v + W$,
	\item $T^{\red} := \pi(T)$,
	\item $\AAA^{\red} = \pi(\AAA)$.
\end{enumerate}
\end{defi}

Note that the set $\AAA_x$ is clearly finite if $x \in T$.

\begin{lem}
The set $T^{\red}$ is an open convex cone in $V^{\red}$.
\end{lem}
\begin{proof}
The set $\pi(T)$ is open as the image of an open set and $\pi$ is open. Furthermore, if $[y,z]$ is an interval in $\pi(T)$, then there exist $y', z' \in T$ with $y = y' +W$, $z = z' + W$. The interval $[y',z']$ is contained in $T$ as $T$ is convex, and we find $\pi([y',z']) = [y,z]$.

Finally, $T^{\red}$ is a cone, since for $y \in T^{\red}$ we find $y' \in T$ with $y = y' +W_x$. Since $T$ is a cone, for every $\lambda > 0$ we have $\lambda y' \in T$, so $\lambda y \in T^{\red}$.
\end{proof}

We gather basic properties of $\AAA^{\red}$.

\begin{lem}
The pair $(\AAA^{\red}, T^{\red})$ is a non-degenerate hyperplane arrangement with chambers $\{\pi(K) \mid K \in \KKK\}$. The chambers are furthermore in one to one correspondence to $\KKK$.
\label{lem:red}
\end{lem}
\begin{proof}
The set $\AAA^{\red}$ is a set of hyperplanes in $V^{\red}$, since $W \subset H$ for all $H \in \AAA$.

Let $H' \in \AAA^{\red}$ with $H' = \pi(H)$ for $H \in \AAA$, then $H' \cap T^{\red} \neq \emptyset$, since $H \cap T \neq \emptyset$.

We show that $\AAA^{\red}$ is locally finite in $T^{\red}$. Let $y \in \pi(T)$, then there exists an $y' \in T$ such that $y = y' + W$. Since $\AAA$ is locally finite in $T$, there exists a neighborhood $U \subset T$ containing $y'$ such that $\{H \in \AAA \mid H \cap U \neq \emptyset\}$ is finite. As $\pi$ maps open sets to open sets, $\pi(U)$ is a neighborhood of $y$ in $T^{\red}$. Now let $z \in H \cap \pi(U)$ for some $H \in \AAA_x^\pi$. Let $H' \in \AAA^{\red}$ with $\pi(H) = H'$ and assume $z \in H' \cap \pi(U)$. Then there exists $z' \in H$ with $\pi(z') = z$ and $z'' \in U$ with $\pi(z'') = z$, thus $z' +w = z''$ for some $w \in W$. Hence $z'' \in H$ as well, and we can conclude that $H \cap U \neq \emptyset$.

We have thus established for $H \in \AAA$
$$
H \cap U \neq \emptyset \Leftrightarrow \pi(H) \cap \pi(U) \neq \emptyset.
$$
Therefore $\{H \in \AAA_x^\pi \mid H \cap \pi(U)\}$ is finite.

For the chambers, note that $\pi(K)$ is a connected component of $T^{\red} \setminus \bigcup_{H \in \AAA^{\red}} H$, in particular, as $K \subset T$, $\pi(K) \subset T^{\red}$, furthermore for a connected component $K'$ of $T^{\red} \setminus \bigcup_{H \in \AAA^{\red}} H$, $\pi^{-1}(K')$ is a connected component of $T \setminus \bigcup_{H \in \AAA} H$. This completes the proof.
\end{proof}

\begin{defi}
Let $(\AAA,T)$ be a hyperplane arrangement, $x \in \ol{T}$. Then set 
	\item $W_x = \bigcap_{H \in \AAA_x} H$,
	\item $V_x = V/W_x$,
	\item $\pi = \pi_x: V \mapsto V_x, v \to v + W_x$,
	\item $T_x = \pi(T)$,
	\item $\AAA_x^\pi = \pi(\AAA_x)$.
\end{defi}

\begin{fol}
The pair $(\AAA_x^\pi, T_x)$ is a non-degenerate hyperplane arrangement with chambers $\{\pi(K) \mid K \in \KKK_x\}$.
\label{cor:redsub}
\end{fol}
\begin{proof}
This follows from Lemma \ref{lem:subarr} and Lemma \ref{lem:red}, since $W_x = \Rad(\AAA_x)$.
\end{proof}

\begin{exmp}
The hyperplane arrangement in Example \ref{exm:notthin} is degenerate, as $\Rad(\AAA) = L$. Reducing this arrangement yields the non degenerate and thin hyperplane arrangement $(\{0\}, \RRR)$.
\label{exm:deg}
\end{exmp}

\subsection{The chamber graph and gated parabolics}

We now consider the structure given by the chambers and their adjacency.

\begin{defi}
Let $(\AAA,T)$ be a hyperplane arrangement. We call two chambers $K, L$ \textit{adjacent}, if $\langle \ol{K} \cap \ol{L} \rangle$ is a hyperplane. If $K,L$ are are adjacent and $\langle \ol{K} \cap \ol{L} \rangle = H$, then we also call $K,L$ to be \textit{adjacent by $H$}.

Define the \textit{chamber graph} $\Gamma = \Gamma(\AAA,T)$ to be the simplicial graph with vertex set $\KKK$, $\{K,L\}$ is an edge if and only if $K$ and $L$ are adjacent.

We call a path in $\Gamma$ connecting $K,L \in \KKK$ a \textit{gallery from $K$ to $L$}. We say a gallery from $K$ to $L$ is \textit{minimal}, if it is of length $d_\Gamma(K,L)$, where $d_\Gamma$ is the distance in $\Gamma$.
\end{defi}

\begin{lem}
Let $K, L \in \KKK$, $v \in K$, $w \in L$. Then $S(K,L) = \sec([v,w])$, and $|S(K,L)|$ is finite.
\label{finsep}
\end{lem}
\begin{proof}
As $T$ is convex, the line $[v,w]$ is contained in $T$, and is compact. Therefore, as $\AAA$ is locally finite, by Lemma \ref{compfin} $\sec([v,w])$ is finite.

Let $H \in \AAA$ and assume $H \cap [v,w]$ is empty, then $D_H([v,w])$ is well-defined and $D_H(K) = D_H([v,w]) = D_H(L)$, so $H$ does not separate $K$ and $L$. We can conclude $S(K,L) \subseteq \sec([v,w])$. On the other hand, let $H \in \sec([v,w])$. Since $v \in K$, $w \in L$, the interval $[v,w]$ is not contained in $H$ and neither $v$ nor $w$ is in $H$. Assume $D_H(\{v\}) = D_H(\{w\})$, since $D_H(\{v\})$ is convex, we obtain $[v,w] \subset D_H(\{v\})$, contradicting our assumption. Therefore $D_H(K) = D_H(\{v\}) = -D_H(\{w\}) = -D_H(L)$. This yields $\sec([v,w]) \subseteq S(K,L)$.
\end{proof}

\begin{lem}
Two chambers $K, L \in \KKK$ are adjacent by $H$ if and only if $S(K,L) = \{\langle H \rangle\}$. In this situation $L$ is the unique chamber adjacent to $K$ by $H$.
\label{adjsep}
\end{lem}
\begin{proof}
Assume first that $K,L$ are adjacent by $H$. So $H = \langle \ol{K} \cap \ol{L} \rangle$. Note that $H \in W^K \cap W^L$ by definition. Since $H \in W^K$, the set $F := \ol{K}\cap \ol{L}$ is a convex cone spanning $H$. Take a point $x$ in the interior of $F \subset H$. For a suitable $\varepsilon >0$ we find $U_\varepsilon(x) \cap H \subset F$ and $\sec(U_\varepsilon(x)) = \{H\}$. The hyperplane $H$ separates $U_\varepsilon(x)$ into two connected components, say $U_+$ and $U_-$. Since $K$ is open, $F \subset \ol{K}$, w.\ l.\ o.\ g.\ $K$ meets $U_+$. Since $\sec(U_\varepsilon(x)) = \{H\}$, indeed we find $U_+ \subset K$. Assuming that $L$ meets $U_+$ implies that $L = K$, but then $K \cap L$ generates $V$, contradicting the assumption of $K$ and $L$ being adjacent. Thus $U_- \subset L$, and $H$ separates $K$ and $L$.

Assume $H' \in \AAA$ such that $H'$ separates $K$ and $L$. Then $\ol{K}\subset \ol{(D_{H'}(K))}$ and $\ol{L} \subset \ol{(D_{H'}(L))}$, in particular $\ol{K} \cap \ol{L} \subset H'$. This implies $H' = H$, therefore $S(K,L) = \{H\}$.

For the other implication assume that $H$ is the unique hyperplane separating $K$ and $L$. In this case, note that $K$ and $L$ are distinct. If $H$ is not in $W^K$, $K$ and $L$ are on the same side of every wall of $K$, thus equal, a contradiction. Therefore $H \in W^K \cap W^L$.

Consider $v \in K$ and $w \in L$. The interval $[v,w]$ meets only $H$ by Lemma \ref{finsep}. Therefore $[v,w] \cap H = \{p\}$, with $p \in \ol{K} \cap \ol{L}$. Since $\AAA$ is locally finite, by Lemma \ref{compfin} we find a neighborhood $U$ of $p$ such that $\sec(U) = \{H\}$. Again $U$ is separated by $H$ into two connected components, one contained in $K$, one in $L$. Therefore $U \cap H$ is contained in $\ol{K} \cap \ol{L}$ and this set generates $H$. Thus $K$ and $L$ are adjacent.

For the second statement assume $L'$ is a chamber adjacent to $K$ by $H$. Then we established $S(K,L') = \{H\}$. If $L \neq L'$, there exists a hyperplane $H'$ separating $L$ and $L'$, and we can assume $D_{H'}(K) = D_{H'}(L)$ and $D_{H'}(K) = - D_{H'}(L')$. Then $H'$ separates $K$ and $L'$, which implies $H' = H$. This contradicts the fact that $D_{H'}(K) = D_{H'}(L)$.
\end{proof}

\begin{lem}
Let $K \in \KKK$, $H \in \AAA \cap W^K$. Then there exists a chamber $L$ adjacent to $K$ by $H$.
\label{lem:adjex}
\end{lem}
\begin{proof}
Take a point $x$ in the interior of $\ol{K} \cap H$ with respect to $H$ and a neighborhood $U$ of $x$ contained in $T$ with $\sec(U) = \{H\}$. Since $H \in \AAA$, $-D_H(K)\cap U$ is contained in a chamber $L$, which is adjacent to $K$ by $H$. 
\end{proof}

In the case where $(\AAA,T)$ is not thin, we need to be able to handle hyperplanes which are not in $\AAA$ but occur as a wall of a chamber.

\begin{lem}
Assume $K \in \KKK$, $H \in W^K$ and $H \notin \AAA$. Then $T \cap H = \emptyset$. In particular for all $K, L \in \KKK$ the set $D_H(K)$ is well defined and $D_H(K) = D_H(L)$ holds.
\label{boundwalls}
\end{lem}
\begin{proof}
Assume $T \cap H \neq \emptyset$, then $H$ separates $T$ into two half spaces. As $H$ is a wall of $K$ and $H \notin \AAA$, we find that $K$ is not a connected component of $T \setminus \bigcup_{H \in \AAA}H$. Therefore $T \cap H = \emptyset$, which proves our claim.
\end{proof}

\begin{lem}
Let $K,L \in \KKK$. Then there exists a minimal gallery of length $|S(K,L)|$ connecting $K$ and $L$.
\label{sepisdis}
\end{lem}
\begin{proof}
By Lemma \ref{finsep} $S(K,L)$ is finite, so let $n=|S(K,L)|$. For $n=0$, we have $D_H(K) = D_H(L)$ for every $H \in A$ and all $H \in W^K \cup W^L$ by Lemma \ref{boundwalls}. This implies that $K$ and $L$ are the same connected component. Let $n=1$, then we obtain the statement from Lemma \ref{adjsep}.

We now use induction on $n$, let $S(K,L) = \{H_1, \dots, H_n\}$. We can sort $S(K,L)$ in a way, such that $H_1 \in W^K$. By Lemma \ref{lem:adjex} and Lemma \ref{adjsep} there exists a unique chamber $K'$ adjacent to $K$ such that $S(K,K')=\{H_1\}$. We will show $S(K',L) = \{H_2, \dots, H_n\}$, then the statement of the Lemma follows by induction. Since we have $S(K,K') = \{H_1\}$, we have $D_{H_i}(K) = D_{H_i}(K')$ for $i = 2, \dots n$. Assuming that there exists an additional $H' \in \AAA$, $H' \neq H_i$ for $i=1, \dots n$ with the property that $H'$ separates $K'$ and $L$, provides a contradiction as $H'$ also separates $K$ and $L$. By induction this yields a gallery of length $n$ from $K$ to $L$.

We show that this gallery is also minimal. For this purpose let $\ol{K}= \ol{K}_0, \ol{K}_1, \dots, \ol{K}_m = \ol{L}$ be an arbitrary gallery connecting $K,L$ and let for $i = 1, \dots, m$ the hyperplane $H_i \in \AAA$ such that $\ol{K}_{i-1} \cap \ol{K}_i \subset H_i$. Note that $H_i$ is unique with that property by Lemma \ref{adjsep}. Assume $H \in S(K,L)$, then there exists an index $0 \leq i < m$ such that $H$ separates $K_i$ and $L$ but not $K_{i+1}$ and $L$. So $D_{H}(K_i) = -D_H(L)$ and $D_H(K_{i+1}) = D_H(L)$, $H$ separates $K_{i}$ and $K_{i+1}$. We obtain $H=H_i$, as $K_i, K_{i+1}$ are adjacent by $H_i$ and $H_i$ is unique with that property by Lemma \ref{adjsep}.

Thus $S(K,L) \subset \{H_1, \dots, H_m\}$ and this yields $|S(K,L)| \leq m$. Above we constructed a gallery of length $|S(K,L)|$, hence the gallery is already minimal.
\end{proof}

\begin{fol}
The graph $\Gamma$ is connected.
\label{cor:gammacon}
\end{fol}

\begin{defi}
Let $(M,d)$ be a connected metric space. Let $x \in M$ and $A \subseteq M$. A point $y \in A$ is called a \textit{gate of $x$ to $A$} or \textit{the projection of $x$ on $A$} if $y \in \sigma(x,z)$ for all $z \in A$. The gate of $x$ to $A$ is uniquely determined, if it exists, and we will then denote it by $\proj_A(x)$. The set $A$ is called gated if every $x \in M$ has a gate to $A$.
\end{defi}

\begin{defi}
Let $x \in \ol{T}$. Define the graph $\Gamma_x = (\KKK_x, E_x)$, where $E_x = \{\{K,K'\} \in E \mid K, K' \in \KKK_x\}$. Then $\Gamma_x$ is a subgraph of $\Gamma$ and is called the \textit{parabolic subgraph at $x$}. A subgraph $\Gamma'$ of $\Gamma$ is called \textit{parabolic}, if it is a parabolic subgraph at $x$ for some $x \in \ol{T}$.
\end{defi}

\begin{rem}
The parabolic subgraphs and subarrangements correspond to the residues in the case where $\AAA$ is a hyperplane arrangement coming from a Coxeter group.
\end{rem}

\begin{lem}
Parabolic subgraphs are connected.
\label{Kxcon}
\end{lem}
\begin{proof}
The graph $\Gamma_x$ corresponds to the chamber graph of the hyperplane arrangement $(\AAA_x, T)$, and is therefore connected by Corollary \ref{cor:gammacon}.
\end{proof}

\begin{lem}
Let $x \in \ol{T}$, $K \in \KKK$. Then there exists a unique chamber $G \in \KKK_x$, such that $D_H(K) = D_H(G)$ for all $H \in \AAA_x$.
\label{starunique}
\end{lem}
\begin{proof}
We prove existence first. Let $K \in \KKK$ and consider the intersection 
$$S := T \cap \bigcap_{H \in \AAA_x} D_H(K).$$
Then $S$ is a chamber of $(\AAA_x,T)$, since it is a non-empty connected component of $T$. By Lemma \ref{lem:subarr} $S$ corresponds to a unique chamber $G$ in $\KKK_x$, and $D_H(G) = D_H(S) = D_H(K)$ holds. 

So let $G \neq G'$ be another chamber in $\KKK_x$. Then there exists an element $H \in \AAA$ such that $H$ separates $G$ and $G'$. By Lemma \ref{starsep} we obtain $H \in \AAA_x$, but this implies $D_H(G') \neq D_H(G) = D_H(K)$. Hence $G$ is unique with the property that $D_H(K) = D_H(G)$ for all $H \in \AAA_x$.
\end{proof}

\begin{prop}
For $X \in \ol{T}$ the set $\KKK_x$ are gated subsets of $(\KKK, d_\Gamma)$.\label{stargate}
\end{prop}
\begin{proof}
Let $x \in \ol{T}$ and $\Gamma_x$ be a parabolic subgraph. Let $K \in \KKK$. By Lemma \ref{starunique} there exists a unique chamber $G_K \in \KKK_F$ such that $D_H(G_K) =D_H(K)$ for all $H \in \AAA_x$. We will prove that $G_K$ is a gate of $K$ to $\KKK_x$. So let $L \in \KKK_x$. If $H \in \AAA$ separates $G_K$ and $L$, we immediately get by Lemma \ref{starsep} that $H \in \AAA_x$. On the other hand we get that if $H' \in \AAA$ separates $G_K$ and $K$, by construction of $G_K$ we find $H' \notin \AAA_x$. Now assume $H \in S(K,L) \cap \AAA_x$, we find that since $H \in \AAA_x$, $H$ does not separate $G_K$ and $K$ and therefore must separate $G_K$ and $L$. Assume on the other hand that $H \in S(K,L) \cap (\AAA \setminus \AAA_x)$, then it cannot separate $L$ and $G_K$, and therefore must separate $G_K$ and $K$. Summarized this yields
\begin{align*}
S(K,L) &= (S(K,L) \cap \AAA_x)\ \dot\cup\ (S(K,L) \cap (\AAA \setminus \AAA_x))\\
&= S(K,G_K)\ \dot\cup\ S(G_K,L).
\end{align*}
We thus find for all $L \in \KKK_x$ that $d_\Gamma(K,L) = d_\Gamma(K,G_K) + d_\Gamma(G_K,L)$. So $G_K$ is indeed a gate for $K$ to $\Gamma_x$.
\end{proof}

\subsection{Restrictions}

We will now consider the structure induced on hyperplanes $H$ by the elements in $\AAA$.

\begin{defi}
Let
$(\AAA,T)$ be a simplicial arrangement. Let $H \leq V$ be a hyperplane in $V$, such that $H \cap T \neq \emptyset$. Set 
$$\AAA^H := \{H' \cap H \leq H\ |\ H' \in \AAA \setminus \{H\}, H' \cap H \cap T \neq \emptyset\},$$
this is a set of hyperplanes in $H$ which have non empty intersection with $T \cap H$, if $T \cap H$ is not empty itself. 

We also define the connected components of $H \setminus \bigcup_{H' \in \AAA^H} H'$ as $\KKK^H$.
\end{defi}

\begin{lem}
The pair $(\AAA^H, T \cap H)$ is a hyperplane arrangement of rank $r-1$ with chambers $\KKK^H$.
\end{lem}
\begin{proof}
By definition for $H' \in \AAA^H$ we find $H' \cap H \cap T \neq \emptyset$. The set $T \cap H$ is an open convex cone in $H$, which is isomorphic to $\RR^{r-1}$. Let $x \in T \cap H$, $U$ a neighborhood of $x$ such that $\sec(U)$ is finite. Then $\sec_{\AAA^H}(U \cap H)$ is also finite, hence $\AAA^H$ is locally finite in $T \cap H$.

By definition $\KKK^H$ are the chambers of this arrangement.
\end{proof}

In the case where $H \in \AAA$, we need some basic relations between chambers in $\KKK$ and in $\KKK^H$.

\begin{lem}
Assume $H \in \AAA$ and let $K \in \KKK$. If $H \in W^K$, then $\ol{K} \cap H$ is the closure of a chamber in $\KKK^H$. For every $K' \in \KKK^H$ there exists exactly two chambers $K_1, K_2 \in \KKK$, such that $H \in W^{K_i}$ and $\ol{K_i} \cap H = \ol{K'}$ for $i = 1,2$.
\label{indfaces}
\end{lem}
\begin{proof}
For the first statement, $\ol{K} \cap H$ is contained in the closure of a chamber $K'$ of $(\AAA^H, T \cap H)$, since its interior with respect to $H$ does not meet any hyperplane. Let $x \in \ol{K'} \setminus (\ol{K} \cap H)$. Then $x \in \ol{L}$ for some chamber $L \in \KKK$. Let $H' \in S(K,L)$. Assume $x \notin H'$, then $x \in D_{H'}(L)$, while $\ol{K} \cap H$ is either contained in $H'$ or in $-D_{H'}(L) = D_{H'}(L)$. A contradiction to $x \in K'$. Thus we find $x \in H'$. Since $x \notin \ol{K}$, there exists a wall of $K$ separating $K$ and $x$. A contradiction, hence $\ol{K'} = \ol{K} \cap H$.

For the second statement, take $x \in K'$, then $x \in H$ and $\supp(\{x\}) = \{H\}$. Choose a neighborhood $U$ of $X$ such that $\sec(U) = \{H\}$, which is possible by Lemma \ref{compfin}. The hyperplane $H$ separates $U$ into two connected components, which are each contained in chambers $K_1, K_2$. By definition $K_1$ and $K_2$ are adjacent by $H$, and by Lemma \ref{adjsep} there can not exist another chamber being adjacent to $K_1$ or $K_2$ by $H$.
\end{proof}

\section{Tits cones and Tits arrangements}
\label{sec:tits}
\subsection{Simplicial cones}

In this section we introduce simplicial cones and collect some basic properties. As before, throughout this section let $V = \RR^r$.

\begin{defi}
Let $\alpha \in V^\ast$ be a linear form, then
\begin{align*}
\alpha^\perp &:= \ker \alpha,\\
\alpha^+ &:= \alpha^{-1}(\RR_{>0}),\\
\alpha^- &:= \alpha^{-1}(\RR_{<0}).
\end{align*}

Let $B$ be a basis of $V^\ast$, then the \textit{open simplicial cone (associated to $B$)} is 
$$
K^B := \bigcap_{\alpha \in B} \alpha^+.
$$
\end{defi}

\begin{remdef}
With notation as above we find
\begin{align*}
\ol{\alpha^+} &= \alpha^{-1}(\RR_{\geq 0}) = \alpha^\perp \cup \alpha^+,\\
\ol{\alpha^-} &= \alpha^{-1}(\RR_{\leq 0}) = \alpha^\perp \cup \alpha^-.
\end{align*}
Let $B$ be a basis of $V^\ast$. We can then define the \textit{closed simplicial cone (associated to $B$)} as 
$$
\ol{K^B} = \bigcap_{\alpha \in B}\ol{\alpha^+}.
$$

We say a cone is \textit{simplicial} if it is open simplicial or closed simplicial.

A simplicial cone can also be defined using bases of $V$. Let $C$ be a basis of $V$, then the open simplicial cone associated to $C$ is 
$$
K^C = \{\sum_{v \in C}\lambda_v v \mid \lambda_v >0\}
$$
and the closed simplicial cone associated to $C$ is 
$$
\ol{K^C} = \{\sum_{v \in C}\lambda_v v \mid \lambda_v \geq 0\}.
$$

Both concepts are equivalent, and it is immediate from the definition that if $B \subset V^\ast$ and $C \subset V$ are bases, then $K^B = K^C$ if and only if $B$ is, up to positive scalar multiples and permutation, dual to $C$.

A simplicial cone $K$ associated to $B$ carries a natural structure of a simplex, to be precise:

$$\SSSS_K := \{\ol{K} \cap \bigcap_{\alpha \in B'} \alpha^\perp \mid B' \subset B\}$$
is a poset with respect to inclusion, which is isomorphic to $\PP(B)$ with inverse inclusion. If $C$ is the basis of $V$ dual to $B$, we find $\SSSS_K$ to be the set of all convex combinations of subsets of $C$, and $\SSSS_K$ is also isomorphic to $\PP(C)$. Moreover, $\{\RR_{\geq 0}c \mid c \in C\}$ is the vertex set of the simplex $\SSSS_K$.

For a simplicial cone $K$, we denote with $B_K \subset V^\ast$ a basis of $V^\ast$ such that $K^{B_K} = K$.
\label{rem:scs}
\end{remdef}

\begin{lem}
An open or closed simplicial cone is convex and has non-empty interior.
\label{cone:simpprop1}
\end{lem}
\begin{proof}
Let $K$ be an open simplicial cone and $B_K \subset V^\ast$ a basis associated to $K$. Then $K$ is convex and open as an intersection of convex and open subsets. Furthermore let $C$ be dual to $B_K$, then we obtain that $\sum_{v \in C} v \in K$, hence $K$ is not empty and thus has non-empty interior.

If $T$ is a closed simplicial cone, it contains the closure of an open simplicial cone, and has non-empty interior.
\end{proof}

\begin{rem}
The common notation for cones introduces \textit{properness} of a cone $K$ as the property of having non-empty interior and being closed, convex, and pointed, the latter meaning that $v, -v \in K \implies v = 0$. Thus all closed simplicial cones are proper. In our context being proper is not of interest, the cones we are dealing with are either convex and open or already simplicial.
\end{rem}

\subsection{Root systems and Tits arrangements}

In the following let $V = \RR^r$ and $T \subseteq V$ be an open convex cone. In this section we establish a notion of Tits arrangements on $T$. The interesting cases the reader may think of are $T = \RR^r$, or a half-space $T = \alpha^+$ for some $\alpha \in V^\ast$.

We can now define our main objects of interest.
\begin{defi}
Let $T \subseteq V$ be a convex open cone and $\AAA$ a set of linear hyperplanes in $V$. We call a hyperplane arrangement $(\AAA, T)$ a \textit{simplicial arrangement (of rank $r$)}, if every $K \in \mathcal{K}(\AAA)$ is an open simplicial cone.

The cone $T$ is the \textit{Tits cone} of the arrangement. A simplicial arrangement is a \textit{Tits arrangement}, if it is thin.
\end{defi}
\begin{figure}[ht]%
\includegraphics[width=0.6\columnwidth]{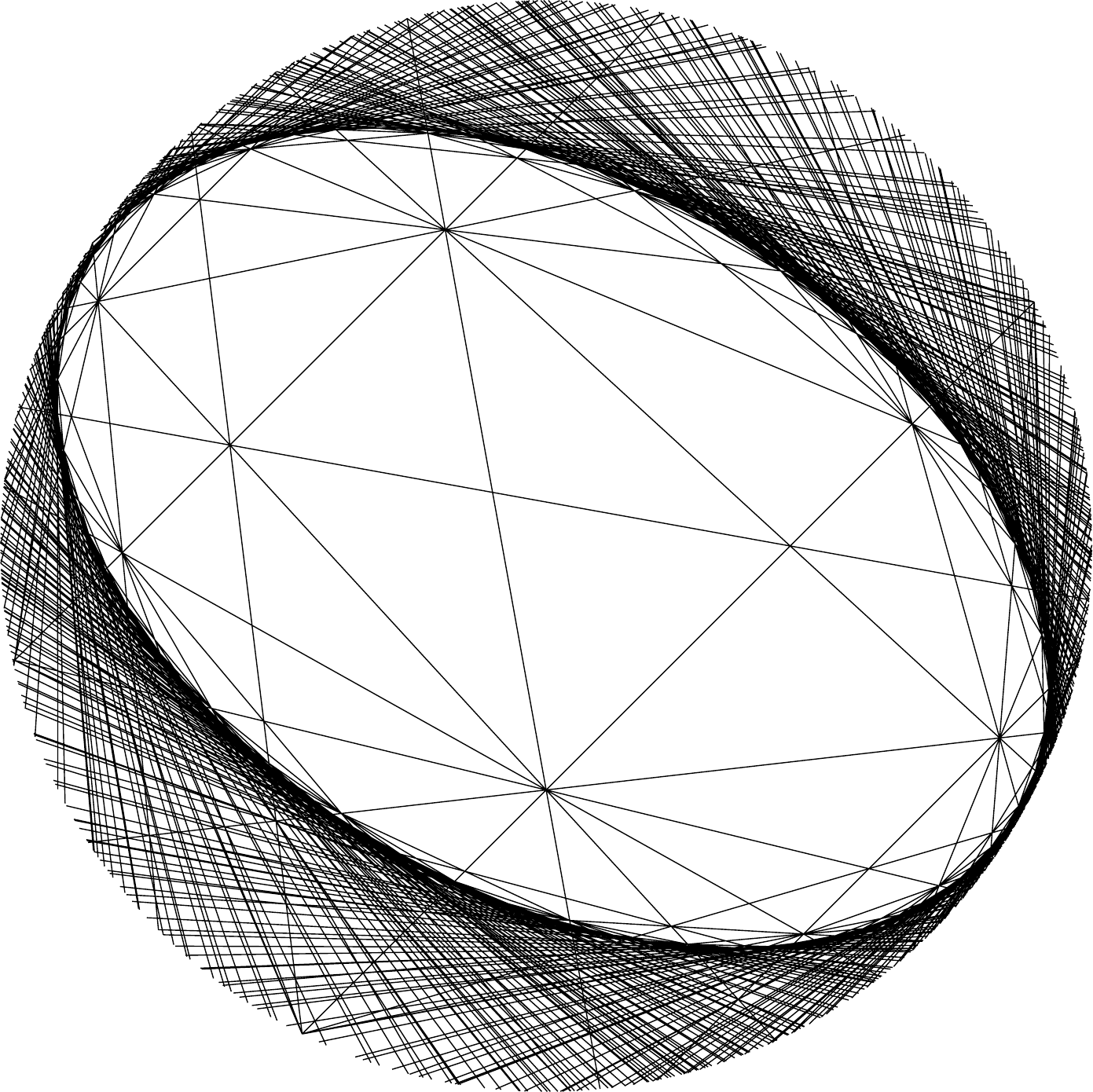}%
\caption{A Tits arrangement of rank three. The area outside of the cone is not entirely black because we display only a finite subset of the hyperplanes.}%
\label{extits}%
\end{figure}
\begin{rem}
\begin{enumerate}
\item The definition of ``thin'' requires that all possible walls are already contained in $\AAA$. If we do not require this, a bounding hyperplane may arise as a bounding hyperplane of $T$ itself. This happens in Example \ref{exm:notthin}, which is not thin. Consider also the case where $T$ itself is an open simplicial cone. Even the empty set then satisfies that the one chamber, which is $T$ itself, is a simplicial cone. However, it has no walls in $\AAA$, hence it is not thin.
\item From the definition it follows that in dimensions $0$ and $1$, there are very few possibilities for simplicial arrangements and Tits arrangements.

In dimension $0$, the empty set is the only possible Tits arrangement.

In the case $V = \RR$, we find $T = \RR$ or $T = \RR_{>0}$ or $T = \RR_{<0}$. In the first case, $\{0_\RR\}$ is a Tits arrangement and in the other two cases, $\{0_\RR\}$ is a simplicial arrangement, but not thin any more.
\item Later in this section we introduce the notion of $k$-spherical arrangements, which is a refinement of being thin.
\item In the case where $T = \RR^r$, the property of $\AAA$ being locally finite is equivalent to $\AAA$ being finite, as $0$ is contained in every hyperplane. Furthermore $T = \RR^r$ is the only case where $0 \in T$, as the cone over every neighborhood of $0$ is already $\RR^r$.
\end{enumerate}
\end{rem}

\begin{exmp}
(i) Consider Example \ref{exm:notthin}, there $K_1$ and $K_2$ are simplicial cones, with $W^{K_1} = \{L, \ker(\alpha_1)\}$ and $W^{K_2} = \{L, \ker(\alpha_2)\}$. The set $\{L\}$ is therefore a simplicial arrangement of rank 2 in $T$, which is finite and hence also locally finite. But $(\{L\}, T)$ is not thin, as
$$T \cap \ker(\alpha_1) = \emptyset = T \cap \ker(\alpha_2).$$
Note that this particular example can be turned into a thin arrangement by factoring out the radical (cp.\ Example \ref{exm:deg}). However, one can easily find examples of non-degenerate simplicial arrangements which are not thin.

(ii) In Example \ref{exm:thinf}, the set $\{L_n, {L_n}' \mid n \in \NN_{>0}\}$ of hyperplanes is not finite, but it is locally finite in $T$. If $(x,y) \in \RR^2$, we cannot find a neighborhood $U$ of $(x,y)$ which meets only finitely many hyperplanes if and only if $x = 0$ or $y = 0$, and hence $(x,y) \notin T$.

A chamber $K$ satisfies one of the following:

\begin{enumerate}
	\item $W^K = \{L_n, L_{n+1}\}$ for some $n \geq 0$, or
	\item $W^K = \{{L_n}', {L_{n+1}}'\}$ for some $n \geq 0$.
\end{enumerate}
Therefore, $(\{L_n, {L_n}' \mid n \in \NN_{>0}\}, T)$ is a Tits arrangement of rank 2.

Likewise, the hyperplane arrangement in Example \ref{exm:univ3} is a Tits arrangement of rank 3 in $T$. The vertices of the simplicial cones are contained in $\partial T$, but all walls meet $T$.

(iii) Figure (\ref{extits}) shows a Tits arrangement coming from the Weyl groupoid of an infinite dimensional Nichols algebra.
\end{exmp}

\begin{lem}
Tits arrangements are non-degenerate.
\label{lem:Titsnd}
\end{lem}
\begin{proof}
Let $K \in \KKK$, $B$ a basis of $V^\ast$ such that $K = K^B$. Since $W^K \subset \AAA$ and $B$ is a basis, we find $\bigcap_{\alpha \in B} \alpha^\perp = \{0\}$.
\end{proof}

Some of the ``classical'' cases are the following choices for $T$.

\begin{defi}
A simplicial arrangement $(\AAA,T)$ is \textit{spherical}, if $T = \RR^r$. We say $(\AAA,T)$ is \textit{affine}, if $T = \gamma^+$ for some $0 \neq \gamma \in V^\ast$. For an affine arrangement we call $\gamma$ the \textit{imaginary root} of the arrangement.
\end{defi}

\begin{rem}
The Tits cone of a Tits arrangement $(\AAA,T)$ resembles the Tits cone for compact hyperbolic Coxeter groups. The geometric representation of an irreducible spherical or affine Coxeter group is a prototype of spherical or affine Tits arrangements.
\end{rem}

\begin{defi}
Let $V = \RR^r$, a \textit{root system} is a set $R \subset V^\ast$ such that
\begin{enumerate}[label=\arabic*)]
  \item $0 \notin R$,
	\item $-\alpha \in R$ for all $\alpha \in R$,
	\item there exists a Tits arrangement $(\AAA,T)$ such that $\AAA = \{\alpha^\perp \mid \alpha \in R\}$. 
\end{enumerate}
If $R$ is a root system, $(\AAA,T)$ as in 3), we say that the Tits arrangement $(\AAA, T)$ is \textit{associated to $R$}.

Let $R$ be a root system. We call a map $\rho: R \to R$ a \textit{reductor of $R$}, if for all $\alpha \in R$
\begin{enumerate}[label=\arabic*)]
  \item $\rho(\alpha) = \lambda_\alpha \alpha$ for some $\lambda_\alpha \in \RR_{>0}$,
	\item $\rho(\langle \alpha \rangle \cap R) = \{\pm \rho(\alpha)\}$.
\end{enumerate}
A root system $R$ is \textit{reduced}, if $\id_R$ is a reductor. Given $R$ and a reductor $\rho$, when no ambiguity can occur, we denote $R^{\red} := \rho(R)$.
\end{defi}

We note some immediate consequences of this definition.

\begin{lem}
Let $R$ be a root system and $(\AAA,T)$ be a Tits arrangement associated to $R$, $\rho$ a reductor of $R$. Then
\begin{enumerate}[label=\roman*)]
	\item $\rho(R)$ is a reduced root system and $(\AAA,T)$ is associated to $R$.
	\item $R$ is reduced if and only if $\langle \alpha \rangle \cap R = \{\pm \alpha\}$ for all $r \in R$.
	\item If $R$ is reduced, $\id_R$ is the only reductor of $R$.
\end{enumerate}
\end{lem}
\begin{proof}
For i), as $\rho(\alpha) = \lambda_\alpha \alpha \neq 0$, we find that $\{\alpha^\perp \mid \alpha \in R\} = \{\rho(\alpha)^\perp \mid \alpha \in R\}$, so $(\rho(R),T)$ is a root system. Due to the properties of $\rho$, we find that $\langle \alpha \rangle \cap \rho(R) = \{ \pm \rho(\alpha)\}$, so $\id_R$ is a reductor of $R$.

For the second statement assume $\langle \alpha \rangle \cap R = \{\pm \alpha\}$ for all $\alpha \in R$, then the identity is a reductor. Assume that the identity is a reductor, then for $\alpha,\beta \in R$ with $\beta = \lambda \alpha$, $\lambda \in \RR$, we find that $\beta \in \{\pm \alpha\}$.

For the third statement assume that $R$ is reduced, so $\id_R$ is a reductor of $(R,T)$. Let $\alpha \in R$, then $\langle \alpha \rangle \cap R = \{\pm \alpha\}$. If $\rho$ is a reductor of $R$, $\rho(\alpha) = \lambda_\alpha \alpha \in R$, as $\lambda_\alpha$ is positive, we find $\lambda_\alpha = 1$ and $\rho = \id_R$.
\end{proof}

\begin{rem}
In the case where for $\alpha \in R$ the set $\langle \alpha \rangle \cap R$ is finite, the canonical choice for $\rho$ is such that $|\rho(\alpha)|$ is minimal in $\langle r \rangle \cap R$; with respect to an arbitrary scalar product.

The notion of a reductor of $R$ is very general, the intention is to be able to reduce a root system even in the case where for $\alpha \in R$ the set $\langle \alpha \rangle \cap R$ does not have a shortest or longest element. In this most general setting, the existence of a reductor requires the axiom of choice.
\end{rem}

Furthermore we find that, given a Tits arrangement $(\AAA,T)$, root systems always exist.

\begin{lem}
Let $(\AAA, T)$ be a Tits arrangement, $\SSS^{r-1}$ the unit sphere in $\RR^r$ with respect to the standard metric associated to the standard scalar product $(\cdot, \cdot)$. Let
$$S = \bigcup_{H \in \AAA} H^\perp \cap \SSS^{n-1}.$$ Then $R = \{(s,\cdot) \in V^\ast \mid s \in S\}$ is a reduced root system for $\AAA$.

Furthermore every reduced root system $R'$ associated to $(\AAA, T)$ is of the form $R' = \{\lambda_\alpha \alpha \mid \lambda_\alpha = \lambda_{-\alpha} \in \RR_{>0}, \alpha \in R \}$, and every such set is a reduced root system associated to $(\AAA, T)$.
\label{rsx}
\end{lem}
\begin{proof}
As $H \in \AAA$ is $r-1$-dimensional, $\dim H^\perp = 1$, so $H^\perp \cap \SSS^{n-1} = \{\pm s\}$ for some vector $s \in V$, with $s^\perp = H$. So $(R,T)$ is a root system associated to $\AAA$. Let $\alpha, \lambda \alpha \in R$ for some $\lambda \in \RR$, then $\alpha^\perp = (\lambda \alpha)^\perp$. As the map $s \mapsto (s, \cdot)$ is bijective, we find $\lambda \in \{\pm 1\}$. Therefore $R$ is reduced.

For the second statement note that the hyperplane $H \in A$ determines $s \in H^\perp$ uniquely up to a scalar. So any $R'$ of the given form satisfies $\{\alpha^\perp \mid \alpha \in R'\} = \AAA$. Furthermore it is reduced, as $\langle \alpha \rangle \cap R' = \{\pm \lambda_\alpha \alpha\}$ for $\alpha \in R$.
\end{proof}

\begin{defi}
Let $(\AAA, T)$ be a Tits arrangement associated to the root system $R$, and fix a reductor $\rho$ of $R$. Let $K$ be a chamber. The \textit{root basis of $K$} is the set 
$$B^K := \{\alpha \in R^{\red} \mid \alpha^\perp \in W^K, \alpha(x) > 0 \text{ for all } x \in K\}.$$
\end{defi}

\begin{rem}
If $(A,T)$ is a Tits arrangement associated to the root system $R$, $K \in \KKK$, then 
$$
W^K := \{\alpha^\perp \mid \alpha \in B^K\}.
$$
Also, as a simplicial cone $K \subset \RR^r$ has exactly $r$ walls, the set $B^K$ is a basis of $V^\ast$. Furthermore, $K^{B^K} = K$.
\end{rem}

The following proposition is crucial for the theory and motivates the notion of root bases. 

\begin{lem}
Let $(\AAA,T)$ be a Tits arrangement associated to $R$, $K$ a chamber. Then $R \subset \pm \sum_{\alpha \in B^K} \RR_{\geq 0}\alpha$. In other words, every root is a non-negative or non-positive linear combination of $B^K$.
\label{posorneg}
\end{lem}
\begin{proof}
The proof works exactly as in the spherical case, see \cite[Lemma 2.2]{Cu11}.
\end{proof}

It is useful to identify Tits arrangements which basically only differ by the choice of basis, therefore we make the following definition.

\begin{defi}
Let $(\AAA,T)$, $(\AAA', T')$ be Tits arrangements associated to the root systems $R$ and $R'$ respectively. Then $(\AAA,T)$ and $(\AAA',T')$ are called \textit{combinatorially equivalent}, if there exists an $g \in \GL(V)$ such that $g\AAA = \AAA'$, $g\ast R = R'$, $g(T) = T'$. Here $\ast$ denotes the dual action of $GL(V)$ on $V^\ast$, defined by $g \ast \alpha = \alpha \circ g^{-1}$.
\end{defi}

\subsection{The simplicial complex associated to a simplicial arrangement}

\begin{remdef}
Let $(\AAA, T)$ be a simplicial arrangement. The set of chambers $\mathcal{K}$ gives rise to a poset
$$
\mathcal{S} := \left\{\ol{K} \cap \bigcap_{H \in \AAA'}H \mid K \in \KKK, \AAA' \subseteq W^K\right\} = \bigcup_{K \in \KKK}\SSSS_K,$$
with set-wise inclusion giving a poset-structure. Note that we do not require any of these intersections to be in $T$. By construction they are contained in the closure of $T$, as every $K$ is an open subset in $T$.
\end{remdef}

We will at this point just note that $\SSSS$ is a simplicial complex. This will be elaborated in Appendix \ref{APP:S}.

\begin{prop}
The poset $\mathcal{S}$ is a simplicial complex.
\label{Ssimpcomp}
\end{prop}

\begin{remdef}
The complex $\mathcal{S}$ is furthermore a chamber complex, which justifies the notion of chambers and is shown in Appendix \ref{APP:SCC}.

We have a slight ambiguity of notation at this point, as a chamber in the simplicial complex $\mathcal{S}$ is the closure of a chamber in $\KKK$. For the readers convenience, a chamber will always be an element $K \in \KKK$, while we will refer to a chamber in $\mathcal{S}$ as a \textit{closed chamber}, written with the usual notation $\ol{K}$. We show in Appendix \ref{APP:SCC} that the closed chambers are indeed chambers in a classical sense.
\end{remdef}

\begin{rem}
\begin{enumerate}
	\item 
Depending on $T$, the above mentioned simplicial complex can also be seen as canonically isomorphic to the simplicial decomposition of certain objects, arising by intersection with the respective simplicial cones.
In case $T = \RR^r$, $\SSSS$ corresponds to a simplicial decomposition of the sphere $\field{S}^{n-1}$. If $T = \alpha^+$ for $\alpha \in V^\ast$, we find $\SSSS$ to be a decomposition of the affine space $\field{A}^{n-1}$, which we identify with the set $\alpha^{-1}(1)$. If $T$ is the light cone and $\AAA$ is CH-like as defined in Definition \ref{def:spherical} below, we can find a corresponding decomposition of $\field{H}^{n-1}$.

\item In the literature (see \cite[Chapter V, \S 1]{Bo02}) the simplicial complex associated to a finite or affine simplicial hyperplane arrangement is defined in a slightly different manner. Let $V$ be a Euclidean space, $\AAA$ a locally finite set of (possibly affine) hyperplanes. Define for $A \in \AAA$, $v \in V \setminus A$ the set $D_H(v)$ to be the halfspace with respect to $H$ containing $v$. Set
$$
v \sim w :\Leftrightarrow w \in \bigcap_{v \in H \in \AAA} H \cap \bigcap_{v \notin H \in \AAA} D_H(v).
$$
Then $\sim$ is an equivalence relation, its classes are called facets. Facets correspond to simplices, and form a poset with respect to the inclusion $F \leq F' \Leftrightarrow F \subseteq \overline{F}$.

One obtains immediately that every point in $V$ is contained in a unique facet. However, when the space is not the entire euclidean space but a convex open cone, it is desirable to consider some of the points in the boundary, as they contribute to the simplicial structure of $\SSSS$. Therefore, we prefer our approach before the classical one.

\item The above approach could also be used to define being $k$-sphe\-ri\-cal (see below) for hyperplane arrangements, which are not necessarily simplicial.
\end{enumerate}
\end{rem}

Using the simplicial complex $\SSSS$, it is now possible to refine the notion of being thin for a simplicial arrangement.

\begin{defi}
A simplicial arrangement $(\AAA,T)$ of rank $r$ is called \textit{$k$-spherical} for $k \in \NN_0$ if every simplex $S$ of $\mathcal{S}$, such that $\codim(S) = k$, meets $T$. We say $(\AAA, T)$ is \textit{CH-like} if it is $r-1$-spherical.
\label{def:spherical}
\end{defi}

\begin{rem}
Immediate from the definition are the following:
\begin{enumerate}
	\item A simplicial hyperplane arrangement of rank $r$ is $0$-spherical, as $K \in \KKK$ is constructed as an open subset of $T$.
	\item The hyperplane arrangement $(\AAA,T)$ is spherical if and only if it is $r$-spherical.
	\item $(\AAA, T)$ is thin and therefore a Tits arrangement if and only if it is $1$-spherical.
	\item As $\mathcal{S}$ is a simplicial complex w.\ r.\ t.\ inclusion, being $k$-spherical implies being $k-1$-spherical for $1 \leq k \leq r$.
	\item Examples of CH-like arrangements are all arrangements belonging to affine Weyl groups, where the affine $r-1$-plane is embedded into a real vector space of dimension $r$, as well as all arrangements belonging to compact hyperbolic Coxeter groups, where $T$ is the light cone. The notion CH-like is inspired by exactly this property for groups, i.\ e.\ being \textbf{c}ompact \textbf{h}yperbolic.
	\item As a generalization of (5), take a Coxeter system $(W,S)$ of finite rank. Then $W$ is said to be $k$-spherical if every rank $k$ subset of $S$ generates a finite Coxeter group. The geometric representation of $W$ then yields a hyperplane arrangement which is $k$-spherical in the way defined above. Therefore, being $k$-spherical can be seen as a generalization of the respective property of Coxeter groups.

\item An equivalent condition for $(\AAA,T)$ to be $k$-spherical, which we will use often, is that every $r-k-1$-simplex meets $T$. This uses just the fact that simplices of codimension $k$ are exactly $r-k-1$-simplices.
\end{enumerate}
\end{rem}

\begin{exmp}
The arrangement in Example \ref{exm:notthin} yields the simplicial complex $\SSSS$, with maximal elements $K_1$, $K_2$, vertices $L \cap \ol{T} = \{(x,x) \mid x \geq 0\}$, $\{(x,0) \mid x \geq 0\}$, $\{(0,x) \mid x \geq 0\}$, and the set $\{0\}$ as the minimal element. Thus $(\{L\}, T)$ is $0$-spherical, but not $1$-spherical, as it is not thin.

The arrangement in Example \ref{exm:thinf} is thin and therefore $1$-spherical and CH-like. The vertices are the rays $L_n \cap \ol{T}$, $L_n' \cap \ol{T}$, the minimal element, which has codimension $2$, is $\{0\}$. Since $0 \notin T$, it is not $2$-spherical and in particular not spherical.

In Example \ref{exm:univ3}, the arrangement $(\AAA, T)$ is thin and therefore 1-spherical, but since all vertices are contained in $\partial T$, it is not 2-spherical and thus not CH-like. Note that the group $W$ from this example is not compact hyperbolic.
\end{exmp}

An important observation is the fact that the cone $T$ can be reconstructed from the chambers.

\begin{lem}
For a simplicial hyperplane arrangement $(\AAA, T)$ we find
$$\ol{T} = \ol{\bigcup_{K \in \KKK} K}.$$
Furthermore $(\AAA,T)$ is CH-like if and only if
$$T = \bigcup_{K \in \KKK} \ol{K}$$
holds.
\label{Tchambers}
\end{lem}
\begin{proof}
As $\bigcup_{K \in \KKK} K \subset T$, the inclusion $\ol{\bigcup_{K \in \KKK} K} \subset \ol{T}$ holds. If $x \in T$, either $x \in K$ for some $K \in \KKK$, or $x$ is contained in a finite number of hyperplanes, thus in a simplex in $\SSSS$ and also in the closure of a chamber, and therefore $x \in \ol{\bigcup_{K \in \KKK} K}$. 

So let $x \in \ol{T}\setminus T$. Assume further that $x$ is not contained in any simplex in $\SSSS$, else the statement follows immediately as above. Therefore $\supp(x) = \emptyset$, and as $T$ is convex, this means $x$ is in the boundary of $T$. So let $U := U_\delta(x)$ be the open $\delta$-ball with center $x$, then $U \cap T \neq \emptyset$. As $U$ is open, $U \cap T$ is open again. The set $\sec(U\cap T)$ can not be empty, else $U_x\cap T$ is not contained in a chamber, in a contradiction to the construction of chambers. So let $\KKK_0$ be the set of chambers $K$ with $K \cap U \neq \emptyset$. If $x$ is not contained in the closure of $\bigcup_{K \in \KKK_0} K$, we find an $\delta > \varepsilon > 0$ such that the open ball $U_\varepsilon(x)$ does not intersect any $K \in \KKK_0$. But $U_\varepsilon(x) \cap T$ must again meet some chambers $K$, which are then also in $\KKK_0$, a contradiction. So $x \in \ol{\bigcup_{K \in \KKK} K}$ and equality holds.

For the second statement, if $(\AAA,T)$ is CH-like, note that every $x \in T$ is contained in some simplex $F \in \SSSS$, so $T \subset \bigcup_{K \in \KKK} \ol{K}$ holds. Now if $x \in \ol{K}$ for some $K \in \KKK$, $x$ is contained in some simplex $F \subset \ol{K}$. Now $\AAA$ is CH-like, so $F$ meets $T$. Boundaries of simplices are simplices, hence the intersection $F \cap (\ol{T} \setminus T)$ is again a simplex in $\SSSS$, being CH-like yields that this intersection is empty, which proves the other inclusion.

The other direction of the second statement is immediate from the definition of being $r-1$-spherical.
\end{proof}

\begin{rem}
For a Tits arrangement $(\AAA,T)$ it is possible to show that we can also describe $T$ as the convex closure of $T_0 := \bigcup_{K \in \KKK}K$, or alternatively as
$$
T = \bigcup_{x,y \in T_0}[x,y].
$$
\end{rem}

We will require the existence of a type function of $\SSSS$ (for the definition, see Definition \ref{def:numb}), which is given by the following proposition, proven in Appendix \ref{APP:SCC}.

\begin{prop}
Let $(\AAA,T)$ be a simplicial arrangement. The complex $\SSSS: = \SSSS(\AAA,T)$ is a chamber complex of rank $r$ with
$$\Cham(\SSSS) = \{\ol{K} \mid K \in \KKK\}.$$
The complex $\SSSS$ is gated and strongly connected. Furthermore there exists a type function $\tau: \SSSS \to I$ of $\SSSS$, where $I = \{1, \dots, r\}$. The complex $\SSSS$ is thin if and only if $(\AAA, T)$ is thin, and $\SSSS$ is spherical if and only if $(\AAA,T)$ is spherical.
\label{prop-Sproperties}
\end{prop}

We will now take a closer look at the relations between the bases of adjacent chambers. Again the proof follows \cite[Lemma 2.8]{Cu11} closely.

\begin{lem}
Let $(\AAA,T)$ be a Tits arrangement associated to $R$, and let $K,L \in \KKK$ be adjacent chambers. Assume $\ol{K} \cap \ol{L} \subset \alpha_1^\perp$ for $\alpha_1 \in B^K$. If $\beta \in B^L$, then either
\begin{enumerate}[label=\roman*)]
	\item $\beta = -\alpha_1$ or
	\item $\beta \in \sum_{\alpha \in B^K}\RR_{\geq 0} \alpha$.
\end{enumerate}
holds.
\label{adjcham1}
\end{lem}
\begin{proof}
By Lemma \ref{posorneg} we can assume $\beta$ is either as in case ii) or $-\beta = \sum_{\alpha \in B^c} \lambda_\alpha \alpha$ with $\lambda_\alpha \in \RR_{\geq 0}$. Using that for arbitrary $\varphi,\psi \in V^\ast$ we have $\ol{\varphi^+} \cap \ol{\psi^+} \subseteq \ol{(\varphi+\psi)^+}$, we get $\ol{K} = \bigcap_{\alpha \in B^K} \ol{\alpha^+} \subset \ol{(-\beta)^+}$. We also observe that $\ol{L} \subset \ol{\beta^+}$, therefore we get $\ol{K} \cap \ol{L} \subset \ol{\beta^+} \cap \ol{(-\beta)^+} = \beta^\perp$.

By choice of $\alpha_1$ and since the elements in $B^K$, $B^L$ are reduced, we find $\alpha_1^\perp = \beta^\perp$ and $\beta = -\alpha_1$.
\end{proof}

\begin{remdef}
Assume for $K \in \KKK$ that $B^K$ is indexed in some way, i.e. $B^K = \{\alpha_1, \dots, \alpha_r\}$.
For any set $I$, define the map $\kappa_I: \mathcal{P}(I) \to \mathcal{P}(I)$ by $\kappa(J) = I \setminus J$. Set $\kappa := \kappa_{\{1, \dots, r\}}$.

For every simplex $F \subset \ol{K}$ there exists a description of the form $F = \ol{K} \cap \bigcap_{\alpha \in B_F}\alpha^\perp$ for some $B_F \subset B^K$ by Remark \ref{rem:scs}, which gives an index set $J_F = \{i \mid \alpha_i \in B_F\}$.

Finally this gives rise to a type function of $\ol{K}$ in $\mathcal{S}$, by taking the map $\tau_{\ol{K}}: F \mapsto \kappa(J_F)$. By Theorem \ref{numbext} the map $\tau_{\ol{F}}$ yields a unique type function $\tau$ of the whole simplicial complex $\SSSS$. So let $L \in \KKK$ be another chamber, then the restriction $\tau|_{\ol{L}}$ is a type function of $\ol{L}$ as well. Assume $B^L = \{\beta_1, \dots, \beta_r\}$, this yields a second type function of $\ol{L}$ in the same way we acquired a type function of $\ol{K}$ before,
$$\tau_{\ol{L}}: F \mapsto \kappa(\{i \mid F \subset \beta_i^\perp\}).
$$
We now call the indexing of $B^L$ \textit{compatible with $B^K$}, if $\tau_{\ol{L}} = \tau|_{\ol{L}}$.

Note that since the type function $\tau$ is unique, there is a unique indexing of $B^L$ compatible with $B^K$. 
\end{remdef}

The existence of the type function and Lemma \ref{adjcham1} allow us to state the following lemma, which gives another characterization root bases of adjacent chambers.

\begin{lem}
Assume that $(\AAA, T)$ is a Tits arrangement associated to $R$. Let $K, L \in \KKK$ be adjacent chambers and choose a indexing $B^K = \{\alpha_1, \dots, \alpha_r\}$. Let the indexing of $B^L = \{\beta_1, \dots, \beta_r\}$ be compatible with $B^K$. Assume $\ol{K} \cap \ol{L} \subset \alpha_k^\perp$ for some $1 \leq k \leq r$. Then $\beta_i \in \langle \alpha_i, \alpha_k \rangle$.
\label{adjcham15}
\end{lem}
\begin{proof}
If $i=k$, this is immediate from Lemma \ref{adjcham1}. So let $i \neq k$ and assume w.l.o.g. $k=1$.

Consider the type function $\tau_{\ol{K}}$ of $\ol{K}$ and $\tau_{\ol{L}}$ of $\ol{L}$ as restrictions of the unique type function $\tau$. Being compatible yields that $\tau(\ol{K}\cap \ol{L}) = \tau(\ol{K} \cap \alpha_1^\perp) = \{2, \dots, r\}$.

Then $\tau(\ol{K}\cap \ol{L} \cap \alpha_i^\perp) = \{2, \dots, r\}\setminus \{i\} = \tau(\ol{K}\cap \ol{L} \cap \beta_i^\perp)$ holds and we get $\tau(\ol{K} \cap \alpha_i^\perp \cap \alpha_1^\perp) = \{2, \dots, n\} \setminus \{i\} = \tau(\ol{d} \cap \beta_i^\perp \cap \beta_1^\perp)$. We conclude that $\beta_i^\perp \cap \beta_1^\perp = \beta_i^\perp \cap \beta_1^\perp$, and as $\alpha_1 = -\beta_1$ we find $\beta_i^\perp \cap \alpha_1^\perp = \alpha_i^\perp \cap \alpha_1^\perp$.

Then we find $\langle \beta_i, \alpha_1\rangle = \langle \alpha_i, \alpha_1 \rangle$, so $\beta_i$ is a linear combination of $\alpha_1$ and $\alpha_i$, which proves our claim.
\end{proof}

\section{Parabolic subarrangements and restrictions of Tits arrangements}
\label{subarr}

\subsection{Parabolic subarrangements of Tits arrangements}

We will note more properties about parabolic subarrangements of Tits arrangements, in particular we will give suitable root systems associated to these subarrangements. Remember from Corollary \ref{cor:redsub} that $(\AAA_x^\pi, T_x)$ is a hyperplane arrangement.

\begin{defi}
Assume $(\AAA,T)$ is a Tits arrangement associated to $R$. Let $x \in \ol{T}$, define $R_x := \{\alpha \in R \mid \alpha^\perp \in \AAA_x\}$. For $K \in \KKK_x$ set $B^K_x := B^K \cap R_x$.
\end{defi}

\begin{lem}
With notation as above, let $K \in \KKK_x$, then $\langle B^{K}_x \rangle = \langle R_x \rangle$. In particular, $B^{K}_x$ is a basis of $\langle R_x \rangle$.
\label{indeqdim}
\end{lem}
\begin{proof}
The inclusion $\langle B^{K}_x \rangle \subset \langle R_x \rangle$ holds due to $B^{K}_x \subset R_x$.

Now the space $(B^{K}_x)^\perp = \{v \in V\ |\ \alpha(v) = 0 \ \forall\ \alpha \in B^{K}_x\}$ is a subspace of $V$ and $(B^{K}_x)^\perp \cap \ol{K}$ is a face of $\ol{K}$ containing $x$. Assume $F$ is a face of $(B^{K}_x)^\perp \cap \ol{K}$ containing $x$. Then $F$ is also a face of $\ol{K}$ and has the structure $F = \ol{K} \cap \bigcap_{i=1}^m \alpha_i^\perp$ for some $\alpha_1, \dots, \alpha_m \in B^K$. Then by definition $x \in \alpha_i^\perp$ for $i = 1, \dots, m$, and $(B^{K}_x)^\perp \cap \ol{K} \subset F$, so $(B^{K}_x)^\perp \cap \ol{K}$ is a minimal face of $\ol{K}$ containing $x$.

The inclusion $\langle B^{K}_x \rangle \subset \langle R_x \rangle$ implies $(R_x)^\perp \subset (B^{K}_x)^\perp$, therefore the set $(R_x)^\perp \cap \ol{K}$ is also a face of $\ol{K}$, furthermore it contains $x$ by definition. We can conclude $(B^{K}_x)^\perp \subset (R_x)^\perp$, which yields the equality.

The second claim follows since $B^{K}_x := B^K \cap R_x$ and the elements in $B^K$ are linearly independent.
\end{proof}

\begin{rem}
Note that $\alpha \in R_x$ defines a map $V_x \to \RR$ by $v + W \mapsto \alpha(v)$, which is well defined since $W \in \ker \alpha$. We will identify this map with $\alpha$ and think of $R_x$ as a subset of $V_x^\ast$.
\end{rem}

Most of the time we are interested only in $R_x$ and its combinatorial properties, therefore it does not matter whether we consider these in $V_x^\ast$ or in $V^\ast$. But formally the transition to $T_x$ is necessary to obtain a root system in the strict sense.

We will consider $V_x$ as a topological space with respect to the topology in $V/W$, which comes from the standard topology on $\RR^r$.

\begin{prop}
Assume $(\AAA, T)$ is a Tits arrangement associated to $R$, $x \in \ol{T}$. The hyperplane arrangement $(\AAA_x^\pi, T_x)$ is simplicial. The chambers of this arrangement correspond to $\KKK_x$.

Moreover, if $x \in T$, then $(\AAA_x^\pi, T_x)$ is spherical. If $x \notin T$, $\dim V_x = m$, and $(\AAA, T)$ is $k$-spherical for some $k \in \NN$, then $(\AAA_x^\pi, T_x)$ is $\min(m,k)$-spherical.

In particular $(\AAA_x^\pi,T_x)$ is a Tits arrangement associated to $R_x$.
\label{indatpoint}
\end{prop}
\begin{proof}
Assume $\dim V_x = m$, so $\dim \langle R_x \rangle = m$ and $\dim W^\perp  = r-m$.

Let $K \in \KKK_x$, $\pi: V \to V_x$ denote the standard epimorphism. Let $\KKK'$ be the connected components of $T_x \setminus \bigcup_{\alpha \in R_x} \alpha^\perp$. The set $\pi(B^{K}_x)$ is a basis for $V_x$ by Lemma \ref{indeqdim} and we find $\pi(K) \subset T_x \cap \bigcap_{\alpha \in B^{K}_x} \pi^\ast(\alpha)^+$ by definition. Denote this intersection by $K'$.

First we show that $K' \in \KKK'$. Assume there exists $\beta \in R_x$ such that there exist $y',z' \in K'$ with  $\beta(y') > 0$, $\beta(z')<0$. Then we find $y, z \in T$ with the properties: $\beta(y) > 0$, $\beta(z)<0$, $\alpha(y)>0, \alpha(z)>0$ for all $\alpha \in B^{K}_x$. For any $0 < \lambda < 1$ and $\alpha \in B^{K}_x$ we find $\alpha(y-\lambda(y-x)) = \alpha(y)(1-\lambda)>0$ and $\alpha(z-\lambda(z-x))<0$ as $\alpha(x) = 0$. So for $0 < \lambda_y, \lambda_z < 1$ the points $y - \lambda_y(y-x)$, $z-\lambda_z(z-x)$ still satisfy the above inequalities.

Now let $\alpha \in B^K\setminus B^{K}_x$, then $\alpha(x)>0$. So choosing $0 < \lambda_y, \lambda_z < 1$ large enough we find that the points $y_1 := y - \lambda_y(y-x)$, $z_1 := z-\lambda_z$ are close enough to $x$ to satisfy $\beta(y_1)>0$, $\beta(z_1)<0$ and $\alpha(y_1)>0< \alpha(z_1)$ for all $\alpha \in B^K$. So $y_1, z_1 \in K$, in contradiction to the simplicial structure of $\SSSS$.

By its definition, $K'$ is a simplicial cone, the arrangement $(\AAA_x^\pi, T_x)$ therefore is simplicial. It remains to determine for which $k$ the pair $(\AAA_x^\pi, T_x)$ is $k$-spherical.

In the case where $x \in T$, $\pi(T) = V_x$ and $\AAA_x^\pi$ is spherical. Construct a simplex $F'$ in the following way:

Since $K$ is a simplicial cone, there exists an closed simplex $S$ with the property $\ol{K} = \RR_{>0}S \cup \{0\}$ (for details, see Remark \ref{rem:chse}). Let $F_x$ denote the minimal face of $S$ such that $x \in \RR_{>0}F_x$, and let $V(F_x)$ be the vertex set of $F_x$. Then the vertices $V(S) \setminus V(F_x)$ span a face of $S$, denote this face by $F'$. The vertices $\pi(V(F'))$ are linearly independent: Assume $\sum_{v \in V(F')} \lambda_v \pi(v) =0$, then $\sum_{v \in V(F')} \lambda_v v \in W^\perp$. Furthermore note that $W^\perp$ is spanned by $V(F_x)$. So we get a linear combination of the form $\sum_{v \in V(F')} \lambda_v v = \sum_{v \in V(F_x)} \lambda_v v$. Linear independence of $V(S)$ yields $\lambda_v = 0$ for all $v \in V(S)$.

This also gives us $|V(F')| \leq m$. Assuming inequality, we find more than $r-m$ vertices in $W^\perp$, a contradiction to $\dim W^\perp  = r-m$. So $\pi(F')$ spans indeed an $m-1$-simplex in $V_x$.

We show $\ol{K'} = \RR_{>0}\pi(F') \cup \{0\}$. In general for every $v \in V(\ol{K})$ there exists a unique $\alpha \in B^K$ such that $\alpha(v) > 0$ by Remark \ref{rem:scs}, all other $\alpha \neq \beta \in B^c$ satisfy $\beta(v) = 0$. Since for $v \in V(F')$ the vector $v$ is not contained in any proper face of $\ol{K}$ containing $x$, we find a unique $\alpha \in B^{K}_x$ such that $\alpha(v) > 0$.

Let $y \in \RR_{>0}\pi(F')$, then $y = \sum_{v \in V(F')}\lambda_v \pi(v)$ with $\lambda_v \geq 0$ for all $v \in V(F')$. Then $\alpha(v) \geq 0$ for all $v \in F'$, $\alpha \in B^{K,x}$ yields $\alpha(\pi(v)) \geq 0$. Therefore $y \in \ol{K'}$.

So assume $0 \neq y \notin \RR_{>0}\pi(F')$, then $y = \sum_{v \in V(F')} \lambda_v \pi(v)$ and there exists a $w \in \pi(V(F'))$ such that $\lambda_w < 0$. Now there exists a unique $\alpha \in B^{K}_x$ such that $\alpha(w) >0$, so we find $\alpha(\sum_{v \in V(F')} \lambda_v v) < 0$. Therefore $\alpha(y) < 0$ holds, and therefore $y \notin \ol{K'}$. We can conclude $\ol{K'} = \RR_{>0}\pi(F') \cup \{0\}$.

Now assume $x \notin T$ and $\AAA$ is $k$-spherical. Let $F'$ be as above, then $\ul{F'}$ is isomorphic as a simplicial complex to the closed chamber $\RR_{>0}\pi(F')) \cup \{0\}$ and all simplices contained therein. A face of $F'$ meets $T$ if and only if a face of $\RR_{>0}\pi(F') \cup \{0\}$ meets $T_x$. Now $F'$ is an $m-1$-simplex, as it is spanned by $m$ vertices, likewise $F_x$ is an $r-m-1$-simplex.

Let $F_1 \subset F'$ be a face of $F'$, and assume $F_1$ is an $l$-simplex. Then $V(F_1) \cup V(F_x)$ generate an $r-m+l$-simplex $F_2$ of $\SSSS$. As $\AAA$ is $k$-spherical, $F_2$ therefore meets $T$ if $r-k-1 \leq r-m+l$ is satisfied, or equivalently $m-k-1 \leq l$. Under this condition also $\pi_x(F_2) = \pi_x(F_1)$ meets $T_x$, we can conclude that $\AAA_x^\pi$ is $k$-spherical. Since $\AAA_x^\pi$ will not be more than $m$-spherical, since it is an arrangement of rank $m$, we find that $\AAA_x^\pi$ is $\min(k,m)$-spherical.

So as $(\AAA,T)$ is thin, $(\AAA_x^\pi, T_x)$ is a Tits hyperplane arrangement with root system $R_x$. Since chambers $K \in \KKK_x$ and in $K' \in \KKK'$ are uniquely determined by the sets $B^{K}_x$ and $B^{K'}$ we find that $\pi$ induces a bijection between $K_x$ and $K'$.
\end{proof}

\begin{rem}
\begin{enumerate}
	\item The Tits arrangement $(\AAA_x^\pi, T_x)$ is not so much dependent on the point $x$ as on the subspace spanned by $x$, as one gets the same arrangement for every $\lambda x$, $0 <\lambda \in \RR$. This hints to the fact that one can see a simplicial arrangement as a simplicial complex in projective space. However, we will not elaborate this further.
	\item It may be worth mentioning that the chambers $\KKK_x$ and the underlying simplicial complex correspond to the star of the smallest simplex in $\SSSS$ containing $x$. 
	\item Its possible to show that for a simplicial hyperplane arrangement $(\AAA,T)$, the arrangement $(\AAA_x^\pi, T_x)$ is also simplicial. However, the existence of a root system simplifies the proof.
	\item When $x \notin T$, $(\AAA_x^\pi,T_x)$ may become $k'$-spherical for some $k' > \min(k, m)$. The reason for this is that $\AAA$ not being $k'$-spherical does not imply that every $r-k'-1$-simplex contained in a chamber in $\KKK_x$ does not meet $T$.
\end{enumerate}
\end{rem}

Note that the above statements make sense if $R_x = \emptyset$, this occurs if and only if either $x \in T$ is in the interior of a chamber, or $x \notin T$ does not meet any hyperplane $H \in \AAA$. However, in this case we have $\langle B^{K}_x \rangle = \{0\}$ and the induced arrangement is the empty arrangement. This is not a problem, since in this case $W^\perp = V$ and $V_x = \{0\}$, but this case is somewhat trivial. Therefore the requirement $R_x \neq \emptyset$ is quite natural to make, and we will assume this from now on.

Another trivial case occurring can be $x = \{0\}$, in which case $R_x = R$, $\AAA_x = \AAA$ and $T_x = T$.

We can also give an exact criterion to when $\AAA_x^\pi$ is a spherical arrangement:

\begin{fol}
Assume $(\AAA,T)$ is a Tits arrangement associated to $R$ with rank $r \geq 2$. Let $x \in \ol{T}$. Then $\AAA_x$ and $R_x$ are finite if and only if $x \in T$.

In particular, a simplicial arrangement is finite if and only if it is spherical.
\label{Axfin}
\end{fol}
\begin{proof}
If $x \in T$, $\AAA_x$, $R_x$ are finite by Proposition \ref{indatpoint}. So assume $\AAA_x$, $R_x$ are finite and let $x \in \ol{T}$. Then also $\KKK_x$ is a finite set, so let $K \in \KKK_{x}$ and by Lemma \ref{Kxcon} we find $K' \in \KKK_x$ such that $d(K,K')$ is maximal. Note that any minimal gallery between chambers in $\KKK_x$ is already in $\KKK_x$, since for $H \in \AAA$ with $D_H(K) = - D_H(K')$ we obtain $H\in \AAA_x$ by Lemma \ref{starsep}.

Let $\dim V_x = m$, then the $m$ adjacent chambers $K^1, \dots, K^m$ of $K'$ exist, since $\AAA$ is thin, and are all closer to $K$ than to $K'$. Let $K^i$ be adjacent by $H_i$ to $K$, we can conclude $D_{H_i}(K) = -D_{H_i}(K')$ for $i=1,\dots, m$. Let $U$ be an open ball with center $x$, and let $U' = U \cap K'$, $U'' = U \cap K$, then $U'$ and $U''$ are open as well and contained in $T$. Take $y' \in U'$, so $y' = x + (y'-x) \in U'$ and $y'' := x -(y'-x)$ is in $U''$. We find $x \in \sigma(y',y'') \subset T$, which shows our first statement.

The last statement is obtained by taking $x= 0_V$.
\end{proof}

We finish this section with two observations.

\begin{lem}
The root system $R$ is reduced if and only if $R_x$ is reduced for every $x \in T$.
\end{lem}
\begin{proof}
This follows immediately since $R_x$ is constructed as a subset of $R$ and $R = \bigcup_{x \in T} R_x$, note that for $\alpha \in R_x$ we find $\alpha^\perp \cap T \neq \emptyset$.
\end{proof}

\begin{lem}
Let $(\AAA,T)$ be a Tits arrangement associated to $R$. Let $x \in \ol{T}$ with $R_x \neq \emptyset$. Let $K,L \in \KKK^x$ be adjacent by $\alpha_1$, and $B^K = \{\alpha_1, \dots, \alpha_n\}$, $B^L = \{\beta_1, \dots, \beta_n\}$ indexed compatible with $B^K$. Then $B^{K}_x \rightarrow B^{L}_x, \alpha_i \mapsto \beta_i$ is a bijection.
\label{restmap}
\end{lem}
\begin{proof}
We know by Lemma \ref{indeqdim} that $\langle B^{K}_x \rangle = \langle B^{L}_x \rangle$. Since $B^{K}_x, B^{L}_x$ consist of linear independent vectors, we get $|B^{K}_x| = |B^{L}_x|$.

It remains to show that the map is well defined. For $\alpha_i \in B^{K}_x$ we find $x \in \beta_i^\perp$. We have $x \in \alpha_1^\perp$ and $x \in \alpha_i^\perp$. Now $\beta_i = \lambda_1\alpha_1 + \lambda_i\alpha_i$ for some $\lambda_1, \lambda_i \in \ZZ$ by Lemma \ref{adjcham15}, so $\beta_i(x) = \lambda_1\alpha_1(x) + \lambda_i\alpha_i(x) = 0$ and we are done.
\end{proof}

\subsection{Restrictions of simplicial arrangements}

In this section, we will discuss how a Tits arrangement $(\AAA,T)$ induces a Tits arrangement on hyperplanes in $\AAA$. In the classical theory of hyperplane arrangements this is also called the restriction of an arrangement [cp.\ \cite{OT92}].

\begin{defi}
Let $(\AAA,T)$ be a Tits arrangement associated to $R$, and let $H$ be a hyperplane. Define
$$\pi_H^\ast: V^\ast \to H^\ast, \alpha \mapsto \alpha|_H,$$
 and set 
$$R^H := \pi_H^\ast(R) \setminus (\{0\} \cup \{\alpha \in \pi_H^\ast(R) \mid \alpha^\perp \cap H \cap T = \emptyset\}).$$
We can also define the connected components of $H \setminus \bigcup_{H' \in \AAA^H} H'$ as $\KKK^H$.
\end{defi}

\begin{rem}
Note that in the case $r=0$ there exists no hyperplane which is not in the arrangement. In the case $r=1$ the set $\AAA^H$ is just the point $\{0\}$ or empty. Our statements will remain true in these cases, but most of the time they will be empty.

In particular we will examine the case where $H \in \AAA$, since otherwise we will not necessarily see an induced simplicial complex.

An interesting special case occurs for affine arrangements with imaginary root $\gamma$ and $H = \gamma^\perp$, since in many cases this yields as $R^H$ a root system associated to a spherical Tits arrangement.
\end{rem}

\begin{lem}
With notation as above, we find $\AAA^H = \{\alpha^\perp \cap H \mid \pi_H^\ast(\alpha) \in R^H\}$.
\end{lem}
\begin{proof}
This follows immediately from the definition.
\end{proof}

\begin{fol}
The elements $K' \in \KKK^H$ are simplicial cones in $H \cap T$.
\label{indsimp}
\end{fol}
\begin{proof}
If $r=0$, the statement is empty, so let $r \geq 1$. Lemma \ref{indfaces} yields that for $K' \in \KKK_H$ the set $\overline{K'}$ is actually an maximal face of some chamber $K \in \KKK$. This is a simplicial cone by Remark \ref{rem:scs}.
\end{proof}

The following observation is immediate from a geometric point of view, but necessary to point out:

\begin{lem}
Let $H \leq V$ be an arbitrary hyperplane, and let $\alpha \in R$. Then we find $\pi_H^\ast(\alpha)^\perp = \alpha^\perp \cap H$, $\pi_H^\ast(\alpha)^+ = \alpha^+ \cap H$ and $\pi_H^\ast(\alpha)^- = \alpha^- \cap H$.
\label{indroots}
\end{lem}
\begin{proof}
First note that $\pi_H^\ast(\alpha)(x) = \alpha(x)$ by definition of $\pi_H^\ast$. The equalities follow immediately by considering the cases $\alpha(x) = 0$ or $\alpha(x) > 0$.
\end{proof}

The above lemma immediately yields the following.
\begin{fol}
$\AAA^H = \{\alpha^\perp \leq H \mid \alpha \in R^H\}$.
\end{fol}

\begin{lem}
Let $(\AAA,T)$ be a Tits arrangement associated to $R$, $K \in \KKK$, $H \in W^K$. Let $B := \pi_H^\ast(B^K) \setminus \{0\}$. Then
\begin{enumerate}[label=\roman*)]
	\item $H \cap \overline{K}= \ol{K'}$ for a unique $K' \in \KKK^H$,
	\item $\langle \ol{K'} \cap \alpha^\perp \rangle = \alpha^\perp$  and $\alpha^\perp \cap K' = \emptyset$ for $\alpha \in B$.
	\item $K' = \{x \in H \mid \alpha(x) > 0 \text{ for all } \alpha \in B\}$.
\end{enumerate}
\label{indbases}
\end{lem}
\begin{proof}
Part i) is clear by the definition of $\KKK$ and $\KKK^H$, since $H \cap \overline{K}$ is a unique maximal face of $\overline{K}$.

For the second statement assume $H = \beta^\perp$ for $\beta \in B^K$. We use that $\ol{K'}$ is a maximal face of $\ol{K}$. The maximal faces of $\ol{K'}$ are exactly the sets of the form $\ol{K} \cap H \cap \alpha^\perp = \ol{K'} \cap \alpha^\perp$ for $\alpha \in B^K \setminus \{\beta\}$, by Lemma \ref{indroots} we obtain that the faces can also be written as $\ol{K'} \cap \alpha^\perp$ for $\alpha \in B$. As the maximal face $\ol{K'} \cap \alpha^\perp$ spans a hyperplane in $H$ contained in $\alpha^\perp$, we conclude $\langle \ol{K'} \cap \alpha^\perp \rangle = \alpha^\perp$ for $\alpha \in B$ and ii) holds.

Assertion iii) is a direct consequence of Lemma \ref{indroots}.
\end{proof}

\begin{prop}
Let $(\AAA,T)$ be a $k$-spherical simplicial arrangement, $k \geq 1$ and $H \in \AAA$. Then $(\AAA^H,T \cap H)$ is a $k-1$-spherical simplicial arrangement. If $(\AAA^H, T \cap H)$ is a Tits arrangement associated to $R$, $R^H$ is a root system for $(\AAA^H, T\cap H)$.
\label{indarrange}
\end{prop}
\begin{proof}
Note that $R^H$ does not contain $0$ by definition, and if $\alpha \in R^H$, we find $\alpha' \in R$ with $\alpha = \pi_H^\ast(\alpha')$, so $-\alpha' \in R$ and $-\alpha \in R^H$.

By Lemma \ref{indroots} we know that $\AAA^H = \{\alpha^\perp \leq H \mid \alpha \in R^H\}$ and by definition we have $\alpha^\perp \cap H \cap T \neq \emptyset$. Furthermore we know that the connected components in $\KKK^H$ are simplicial cones by Corollary \ref{indsimp}.

Let $K' \in \KKK^H$ and $K \in \KKK$ such that $K'$ is a face of $K$. Let $\beta \in B^K$ such that $\beta^\perp = H$. By ii) in Lemma \ref{indbases} we find $W^{K'} = 
\{\pi_H^\ast(\alpha)^\perp \mid \alpha \in B^K \setminus \{\beta\}\}$, and together with iii) in Lemma \ref{indbases} we find reduced roots $B^{K'} = \{\lambda_\alpha \alpha \in (R^H)^{\red} \mid \alpha \in \pi_H^\ast(B^K) \setminus \{0\}, \lambda_\alpha \in \RR_{>0}\}$ with the property $K' = \{x \in H \mid \alpha(x) > 0 \text{ for all } \alpha \in B^{K'}\}$. Hence $\AAA^H$ is a simplicial arrangement in $T \cap H$, and $\AAA^H = \{\alpha^\perp \mid \alpha \in R^H\}$ holds. So if $\AAA^H$ is thin, $R^H$ is a root system for $\AAA^H$.

Now assume $F$ is an $m$-simplex in the simplicial complex $\mathcal{S}^H$ associated to the simplicial arrangement $(\AAA^H, T\cap H)$. Since this is a subset of $\mathcal{S}$, $F$ meets $T \cap H$ if and only if $F$ meets $T$, and therefore $(\AAA^H,T \cap H)$ is $k-1$-spherical if $(\AAA,T)$ is $k$-spherical.
\end{proof}

\begin{fol}
Assume that $(\AAA^H, T \cap H)$ is a Tits arrangement associated to $R^H$. Let $K \in \KKK$ such that $H \in W^H$. Then the set $\pi_H^\ast(B^K) \setminus\{0\}$ is a basis of $H^\ast$.
\label{projbasis}
\end{fol}
\begin{proof}
By Proposition \ref{indarrange} the set $B^{K'}$ is a basis of $H^\ast$, by Lemma \ref{indbases} iii) we find that the elements in $\pi_H^\ast(B^K) \setminus \{0\}$ are non zero scalar multiples of $B^{K'}$.
\end{proof}

\begin{defi}If $(\AAA, T)$ is a simplicial arrangement, we call $(\AAA^H,T\cap H)$ for $H \in \AAA$ as above the \textit{induced simplicial hyperplane arrangement (by $(\AAA,T)$) on $H$} or the \textit{restriction of $(\AAA,T)$ to $H$}.
\end{defi}

\appendix

\section{Simplicial complexes}
\label{APP:simpcomp}

The notation in this chapter is mostly taken from \cite[Appendix A]{Da08} and \cite{Ti74}. We recall the notion of partially ordered sets:

\begin{defi}
Let $(M, \leq)$ be a poset. For $m \in M$ we write
$$\ul{m} := \{m' \in M \mid m' \leq m\}.
$$

Let $(M, \leq)$, $(N, \subseteq)$ be posets. A \textit{morphism of posets} is a map $\varphi: M \rightarrow N$, such that for $a,b \in M$ we have
$$a \leq b \implies \varphi(a) \subseteq \varphi(b).$$
It is called an isomorphism if it is bijective and $\varphi^{-1}$ is a morphism as well.
\end{defi}

\begin{remdef}\label{def:simpcomp}
There are two different approaches to define what a simplex is, since we want to use both, we will introduce them here. 

A \emph{simplex} can be seen as a poset $(S, \leq)$ isomorphic to $(\PP(J), \subseteq)$ for some set $J$, where $\subseteq$ denotes the set-wise inclusion. 
A \emph{simplicial complex} is then a poset $(\Delta, \leq)$ such that
\begin{enumerate}[label=\arabic*)]
	\item $\underline{a}$ is a simplex for all $a \in \Delta$,
	\item $a, b \in \Delta$ have a unique greatest lower bound, denoted by $a \cap b$.
\end{enumerate}

A \textit{vertex} of $\Delta$ is an element $v \in \Delta$ such that $a \leq v$ and $a \neq v$ imply $a = \emptyset$.

Another way to define a simplicial complex is to take a set $J$, and let $\Delta \subset \PP(J)$. Then $(\Delta, \subseteq)$ is a poset. It is a simplicial complex if furthermore for $a \in \Delta$ also $\PP(a) \subseteq \Delta$. The set of vertices corresponds to the set $J$.
The first definition makes it easier to describe the star of a simplex, which will often be useful.

So let $(\Delta, \leq)$ be a simplicial complex. For $a,b \in \Delta$ we say that $a$ \textit{is a face of} $b$ if $a \leq b$. We will write $a < b$ if $a \leq b$ and $a \neq b$.

Due to the second property, there exists a unique minimal element in $\Delta$ which is denoted by $\emptyset$.
\end{remdef}

\begin{defi}
Let $a \in \Delta$, then the \textit{rank of} $a$, $\rk(a)$ is the cardinality of the set of vertices contained in $a$. We define the rank of $\Delta$, $\rk(\Delta) := \sup_{a \in \Delta} \rk(a)$.

For $a \in \Delta$ define the \textit{star of $a$} as $\St(a) := \{b \in \Delta\ |\ a \leq b \}$. This is again a simplicial complex with minimal element $a$.

A \textit{chamber} of $\Delta$ is a maximal element in $\Delta$, we will denote the set of chambers as $\Cham(\Delta)$ or $\mathcal{K}$, if $\Delta$ is unambiguous.

Let $\alpha: \Delta \rightarrow \Delta'$ be a map between simplicial complexes $\Delta$ and $\Delta'$. Then $\alpha$ is called a \textit{morphism of simplicial complexes} if it is a morphism of posets and furthermore $\alpha|_A: \ul{A} \to \ul{\alpha(A)}$ is an isomorphism for all $A \in \Delta$.

A subcomplex $\Delta'$ of $\Delta$ is a subset of $\Delta$ such that the inclusion $\Delta' \rightarrow \Delta$ is a morphism of simplicial complexes.

For $a \leq b$, the \textit{codimension of $a$ in $b$} is the rank of $b$ in $\St(a)$, denoted by $\codim_b(a)$.

We say that $a$ is a \textit{maximal face} of $b$, if $\codim_b(a) =1$.
\end{defi}

\begin{rem}
For simplices occurring as subsets of $\RR^n$ it is convenient to consider the dimension of a simplex rather than the rank. We will denote an $n$-dimensional simplex simply as an $n$-simplex. Therefore, an $n$-simplex will be of rank $n+1$.
\end{rem}

\section{The simplicial complex $\SSSS$}
\label{APP:S}

Fix a simplicial arrangement $(\AAA,T)$. We will show that $\mathcal{S}$ is actually a simplicial complex and furthermore a chamber complex with set of chambers $\KKK$. We already showed in Remark \ref{rem:scs} that the simplicial structure on a closed chamber $\ol{K}$ is induced from the simplicial structure of $\ol{S}$, where $S$ is an open simplex such that $K = \RR_{>0}S$.

Recall the notion of the convex hull of a set.

\begin{remdef}
Let $(V,d)$ be a connected metric space. For an arbitrary subset $X \subset V$ the convex hull of $X$ is the smallest convex set $Y \subset V$, such that $X \subset Y$. For a different approach, define the segment between $x,y \in M$ to be 
$$
\sigma(x,y) := \{z \in M \mid d(x,z) + d(z,y) = d(x,y)\}.
$$
This can be used to define
$$H(X) := \bigcup_{x,y \in X} \sigma(x,y).$$
We can then recursively define
\begin{align*}
H^{(0)}(X) &:= X,\\
H^{(n)}(X) &:= H(H^{(n-1)}(X)) \text{ for } 1 \leq n \in \NN.
\end{align*}
Then the convex hull of $X$ is the set $\bigcup_{n \in \NN}H^{(n)}(X)$.

For $V = \RR^r$ and a linearly independent set $X \subset \RR^r$ the convex hull of $X$ can be more easily described as
$$\{\sum_{x \in X}\lambda_x x \mid 0 \leq  \lambda_x \leq 1 \text{ for all }x \in X, \sum_{x \in X}\lambda_x = 1 \}.$$
In this setting, we will refer to the set 
$$\{\sum_{x \in X}\lambda_x x \mid 0 < \lambda_x < 1 \text{ for all }x \in X, \sum_{x \in X}\lambda_x = 1\}$$
as the \textit{open convex hull of $X$}.

If $X \subset \RR^r$ is linearly independent, then its convex hull $S$ is exactly a simplex of rank $r-1$, with vertex set $V(S) = X$. The simplicial structure coincides with the poset $(\PP(X), \subseteq)$.

Moreover, if $C$ is a basis of $\RR^r$, and $S$ is the open convex hull of $C$, the cone
$$
\RR_{>0}S := \{\lambda v \mid \lambda \in \RR_{>0}, v \in S\}
$$
coincides with $K^C$ as defined in Remark \ref{rem:scs}. If $F$ is a face of $S$, there exists a subset $C_F \subset C$ such that $F$ is the convex hull of $C_F$. Then $\ol{\RR_{>0}F}$ is a face $\ol{K}$ in the simplicial complex $\ol{\ul{K}}$. Therefore $\PP(C)$, $S$, and $\ol{\ul{K^C}}$ are all isomorphic simplices, via the isomorphisms
\begin{align*}
\ul{S} &\to \ol{\ul{K}}, F \mapsto \ol{\RR_{>0}F},\\
\ul{S} &\to \PP(C), F \mapsto C_F.
\end{align*}
\label{rem:chse}
\end{remdef}

\begin{lem}
Let $C = \{v_1, \dots, v_r\}$, $C' = \{v_1', \dots, v_r'\}$ be bases of $V$ such that $K^{C} = K^{C'}$. Then, up to permutation, $\alpha_i = \lambda_i\alpha_i'$ for some $\lambda_i \in \RR_{>0}$ for all $1 \leq i \leq r$.

The same holds for two bases $B,B'$ of $V^\ast$ such that $K^B = K^{B'}$.
\label{equalcones}
\end{lem}
\begin{proof}
Let $S$ be open convex hull of $C$. Then $\RR_{>0}S = K^C = K^{C'}$. One can choose positive scalars $\mu_i \in \RR_{>0}$ such that $w_i = \mu_iv_i' \in \ol{S}$. Then $w_i = \sum_{k=1}^r\kappa_{ik} v_i$ with $\sum_{k=1}^r \kappa_{ik}=1$, $0 \leq \kappa_{ik} \leq 1$. Let $C'' :=\{w_1, \dots, w_r\}$, we find $K^C = K^{C''}$. Hence
$v_i = \sum_{j=1}^r \nu_{ij}w_i$ with $\sum_{j=1}^r \nu_{ij}=1$ and $0 \leq \nu_{ij} \leq 1$.

Therefore the matrix $M_{C}^{C''}(\id)$ describing the base change is non-negative, and the same holds for its inverse $M_{C''}^{C}(\id)$. It is well known (see Theorem 4.6 in Chapter 3 of \cite{BP79} for example) that the inverse of a non-negative matrix is non-negative if and only if the matrix is monomial. Adding the fact that the sum of every column adds up to $1$, we get that $M_{C}^{C''}(\id)$ and $M_{C''}^{C}(\id)$ are already permutation matrices.

The statement for $B, B' \subset V^\ast$ just follows from Remark \ref{rem:scs} by considering the dual basis.
\end{proof}

\begin{rem}
For the following statement recall from \ref{rem:scs} that for a basis $B$ of $V^\ast$ the cone $K^B \subset V$ is given by $$K^B = \bigcap_{\alpha \in B} \alpha^+.$$
For a basis $C$ of $V$ the cone $K^C$ is given by
$$K^C = \{\sum_{v \in C}\lambda_v v \mid 0 < \lambda_v\}.$$
\end{rem}

\begin{fol}
Let $K = K^C = K^B$ for a basis $C$ of $V$ and a basis $B$ of $V^\ast$. Let $\beta \in V^\ast$. Then $\beta(v) \geq 0$ for all $v \in C$ if and only if $\beta \in \sum_{\alpha \in B} \RR_{\geq 0} \alpha$. Likewise $\beta(v) \leq 0$ for all $v \in C$ if and only if $\beta \in -\sum_{\alpha \in B} \RR_{\geq 0} \alpha$.\label{purelc}
\end{fol}
\begin{proof}
Let $\beta = \sum_{\alpha \in B} \lambda_\alpha \alpha$. By Remark \ref{rem:scs} we find that $B$ is dual to $C$ up to positive scalar multiples. Denote with $\alpha_v \in B$ the dual to $v \in C$. So $\beta = \sum_{v \in C} \lambda_v \alpha_v$. Applying this to $C$ yields $\beta(v) =  \lambda_{v}$ for all $v \in C$. This immediately yields both equivalences.
\end{proof}

\begin{lem}
Let $F \in \SSSS$. For every $H \in \AAA$, either $F \subset H$ or $F$ is contained in a unique closed half space of $H$, denoted by $\ol{D_H(F)}$. Furthermore $F \cap H \in \SSSS$, and if $F \in \ol{\ul{K}}$ for some $K \in \KKK$, then $F \cap H \in \ol{\ul{K}}$.
\label{FacIntWal}
\end{lem}
\begin{proof}
In case $F \subset H$, there is nothing to show, so assume $F \not \subset H$. Then the first statement is an immediate consequence of Corollary \ref{purelc} and the fact that the elements $K$ are defined as connected components of $V \setminus \bigcup_{H \in \AAA} H$.
Let $\alpha \in V^\ast$ such that $H = \alpha^\perp$ and $D_H(F) = \alpha^+$.

Let $F$ be the convex hull of the vertices $\RR_{>0}v_1, \dots, \RR_{>0}v_k$, where $k \geq 1$ as we assume $F \not \subset H$. For the same reason we can assume $v_i, \dots, v_k \notin H$ for some $i < k$, and without loss of generality we can assume $v_1, \dots, v_{j-1} \in H$. Then $F \cap H$ is the convex hull of $\RR_{>0}v_1, \dots, \RR_{>0}v_{i-1}$. By Remark \ref{rem:scs} $\ol{\ul{K}}$ is a simplicial complex, hence $F \cap H$ is a simplex in $\SSSS$ contained in $\ol{\ul{K}}$.
\end{proof}

\textit{Proof of Proposition \ref{Ssimpcomp}.}
Let $F \in \SSSS$, then $F = \ol{K} \cap \bigcap_{H \in \AAA_1}H$ for some $K \in \KKK$, $\AAA_1 \subset W^K$. In particular $\ul{F}$ carries the structure of a simplex, since by Remark \ref{rem:scs} $\ol{\ul{K}}$ is a simplex.

So let $F' \in \SSSS$, we have to show $F \cap F' \in \SSSS$. The intersection $F \cap F'$ is not empty, as it contains $0_V$.

Assume $F' = \ol{K'} \cap \bigcap_{H \in \AAA_2} H$, for $K' \in \KKK$, $\AAA_2 \subset W^{K'}$. In the case $K = K'$ there is nothing to show, as $\ul{\ol{K}}$ is a simplicial complex by Remark \ref{rem:scs}. So from now on let $K \neq K'$.

Consider the case that $\AAA_2 = \emptyset$, we find $F' = \ol{K'}$. The set $\ol{K'}$ can be written as $\ol{K'} = \bigcap_{H \in W^{K'}} \ol{D_H(K')}$ by Remark \ref{rem:scs}. By Lemma \ref{FacIntWal} the intersection $F \cap \ol{D_H(K')}$ for $H \in W^{K'}$ is either $F$, in case that $D_H(K')$ contains $F$, or equals $F \cap H$. In both cases, it is again a simplex in $\ol{\ul{K}}$. We can conclude that $F \cap \ol{K'} \in \SSSS$ for every $F \in \SSSS$.

Now let $\AAA_2 \neq \emptyset$, then $F \cap \bigcap_{H \in \AAA_2}H$ is a simplex in $\ol{\ul{K}}$ by Lemma \ref{FacIntWal}.

But $F \cap F'$ can be written as $F \cap \bigcap_{H \in \AAA_2}H \cap \ol{K'}$, using the previous part of the proof shows our claim.

\section{$\SSSS$ as a gated, numbered chamber complex}
\label{APP:SCC}

In this section we recall definitions and basic facts regarding chamber complexes and type functions following \cite{Mu94}.

\begin{defi}
Let $\Delta$ be a simplicial complex. We call $\Delta$ a chamber complex if it satisfies:

\begin{enumerate}[label=\arabic*)]
	\item Every $A \in \Delta$ is contained in a chamber.
	\item For two chambers $C,C' \in \Delta$ there is a sequence
	$$C = C_0, C_1, \dots, C_k = D$$
	such that
	$$
	\codim {}_{C_{i-1}}(C_{i-1} \cap C_i) = \codim {}_{C_{i}}(C_{i-1} \cap C_i) \leq 1
	$$
	for $1 \leq i \leq k$.
\end{enumerate}
We call a sequence as in 2) a \textit{gallery (from $C$ to $D$)} and $k$ its \textit{length}.
\end{defi}

Note that the first property is always satisfied if $\rk(\Delta)$ is finite. In this case it is easy to see (cp.\ \cite[1.3, p.15]{Mu94}) that every chamber has the same rank. A consequence from 2) is that every element in a gallery is again a chamber.

\begin{defi}
For two chambers $C,D \in \Delta$ and $A \in \Delta$ with $A \leq C$ and $A \leq D$ we have $\codim_C(A) = \codim_D(A)$ (cp.\ \cite[1.3]{Ti74}). This allows us to define the \textit{corank of $A$} as $\corank(A) = \codim_C(A)$ for any chamber $C$ containing $A$. We call $C$ and $D$ \textit{adjacent}, if $\corank(C \cap D) = 1$.

For $A \in \Delta$ with $\corank(A) = 1$, we call $\Cham(\St(A))$ a \textit{panel} of $\Delta$.

The complex $\Delta$ is meagre (resp.\ thin, firm, thick), if every panel contains at most two (exactly two, at least two, at least three) chambers.

A chamber complex $\Delta$ is \textit{strongly connected}, if $\St(A)$ is a chamber complex for every $A \in \Delta$.
\end{defi}

\begin{defi}
Let $\Delta$ be a chamber complex and $I$ be an index set. A \textit{type function} of $\Delta$ is a morphism of chamber complexes $\tau: \Delta \to \PP(I)$.

A \textit{weak type function} of $\Delta$ is a family of type functions
$$(\tau_C: \ul{C} \to \PP(I))_{C \in \Cham(\Delta)}$$ 
which is compatible in the sense that $\tau_C|_{\ul{C}\cap\ul{D}} = \tau_D|_{\ul{C}\cap\ul{D}}$ for adjacent chambers $C$ and $D$.
\label{def:numb}
\end{defi}

\begin{rem}
A type function $\tau$ with index set $I$ induces a weak type function $(\tau|_{\ul{C}}:\ul{C} \to \PP(I))_{C \in \Cham(\Delta)}$. Conversely, we show in Lemma \ref{wnumisnum} that a weak type function $(\tau_C:\ul{C} \to \PP(I))_{C \in \Cham(\Delta)}$ on a strongly connected chamber complex gives rise to a type function $\tau$ such that $\tau|_{\ul{C}} = \tau_C$.
\end{rem}

We also recall some notions in metric spaces and the definition of gated subsets, following \cite[1.5.3]{Mu94}.

The following lemma is immediate from the definitions.

\begin{lem}
Let $\Delta$ be a chamber complex, $K,K'$ chambers. Define $d_\Delta(K,K')$ to be the length of a minimal gallery from $K$ to $K'$. Then $(\Cham(\Delta), d_\Delta)$ is a connected metric space.
\end{lem}

We recall the following proposition.

\begin{theo}[see {\cite[1.5.3]{Mu94}}]
Let $\Delta$ be a chamber complex, such that all sets $\Cham(\St(x))$ are gated for $F \in \Delta$ with $\codim_\Delta(F) \in \{1,2\}$. Let $C$ be a chamber and $\tau$ be type function of $\ul{C}$, there exists a unique weak type function $(\tau_D)_{D \in \Cham(\Delta)}$ such that $\tau_C = \tau$.
\label{gsimpwn}
\end{theo}

From now on, let $V = \RR^r$ and $(\AAA,T)$ be a simplicial arrangement, with the respective simplicial complex $\SSSS$.

The set $T$ itself is a metric space as a convex open subset of $\RR^r$. We will for the rest of the chapter denote this metric as $d_T$, and the more metric on the chambers as $d_\SSSS$.

Remember that $\AAA$ is locally finite in $T$, which means that if we take a compact (w.\ r.\ t. $d_T$) subset $X \subseteq T$, the set $\sec(X)$ is finite.

At this point we can justify the notion of the chambers $\KKK$.

\begin{lem}
$\Cham(\SSSS) = \{\ol{K}\ |\ K \in \KKK\}$.
\label{chamchar}
\end{lem}
\begin{proof}
By definition every maximal element in $\SSSS$ is of the form $\ol{K}$ for some $K \in \KKK$. Assume $\ol{K} \subset \ol{K'}$ for $K,K' \in \KKK$, then $\ol{K}$ is contained in some $H \in \AAA$ by Remark \ref{rem:scs}, which contradicts to the definition of $\KKK$. This proves $\Cham(\SSSS) = \{\ol{K}\ |\ K \in \KKK\}$.
\end{proof}

To prove that $\SSSS$ is already a chamber complex, we need a bit more information about the distance between two chambers, depending on the number of hyperplanes separating them.

\begin{lem}
Let $\ol{K} \in \Cham(\SSSS)$, then $\ol{K} = \bigcap_{H \in W^K} D_H(K)$.
\label{chamchar2}
\end{lem}
\begin{proof}
This is a direct consequence of the fact that $\ol{K}= \{x \in T \mid \alpha(x) \geq 0 \text{ for all } \alpha \in B_K\}$ and $W^K = \{r^\perp \mid r \in B_K\}$ for some basis $B_K \subset V^\ast$. For $\alpha \in B_K$ with $\alpha^\perp = H \in W^K$ we find $D_H(K) = \{v \in T \mid \alpha(v) \geq 0\}$.
\end{proof}

\begin{lem}
Two chambers are adjacent in the chamber graph if and only if they are adjacent in the chamber complex. Hence galleries in $\Gamma$ correspond to galleries in $\Cham(\SSSS)$.
\end{lem}

\begin{remdef}
Two closed chambers $\ol{K},\ol{L}$ are adjacent, if $\codim_{\ol{K}}(\ol{K} \cap \ol{L}) = 1 = \codim_{\ol{L}}(\ol{K} \cap \ol{L})$. With respect to Lemma \ref{adjsep} we find that $K,L$ are adjacent if and only if $(K,L)$ is a unique minimal gallery from $K$ to $L$.
\end{remdef}

\begin{defi}
For a simplex $F \in \SSSS$ we set
$$\KKK_F := \{K \in \KKK\ |\ F \subset \ol{K}\}$$
and
$$\AAA_F := \{H \in \AAA\ |\ F \subset H\}.$$
With this notation $\Cham(\St(F)) = \{\ol{K} \mid K \in \KKK_F\}$ holds.
\end{defi}

\begin{prop}
The simplicial complex $\SSSS$ is a strongly connected chamber complex.
\label{constrcon}
\end{prop}
\begin{proof}
The complex $\SSSS$ is a chamber complex, since every simplex is contained in a chamber and two chambers $K,L \in \KKK$ are connected by a gallery of length $|S(K,L)|$ by Lemma \ref{sepisdis} and $|S(K,L)|$ is finite by Lemma \ref{finsep}. For two elements $K,L \in \KKK$ we can therefore define the distance $d_\SSSS(K,L)$ as the length of a minimal gallery connecting $K,L$.

Let $F$ be a simplex in $\SSSS$, and consider the simplicial complex $\St(F)$ with chambers $\KKK_F$. Let $K,L \in \KKK_F$ and assume $d_\SSSS(K,L) = n = |S(K,L)| \geq 1$, so $K \neq L$. The fact that $K,L \in \KKK_F$ implies $F \in \ol{K} \cap \ol{L}$.

We need to show that there exists a gallery in $St(F)$ from $K$ to $L$, which we do by induction on $d_\SSSS(K,L)$. For $d_\SSSS(K,L) = 1$ we have that $K,L$ are adjacent. So let $d_\SSSS(K,L) = n$ and assume $K' \in \KKK$ with the properties that $K,K'$ are adjacent, $K \cap K' \subset H_1$ and $S(K,L) = \{H_1, \dots, H_n\}$. Then $H_1 \in \AAA_F$ by Lemma \ref{starsep}, and $F \in \ol{K} \cap H_1$ implies $F \subset \ol{K} \cap \ol{K'}$. In particular we get $\ol{K'} \in \KKK_F$ and by induction there exists a gallery from $K'$ to $L$ in $\KKK_F$, so we are done.
\end{proof}

With respect to Lemma \ref{stargate}, the following Lemma yields that the sets $\Cham(\St(F))$ are gated for all $F \in \SSSS$.

\begin{lem}
For every $F \in \SSSS$, there exists an $x \in \ol{T}$ such that $F$ is the minimal simplex in $\SSSS$ which contains $x$.
\end{lem}
\begin{proof}
Let $F \in \SSSS$, then there exists $K \in \KKK$ and $\AAA' \subset W^K$ such that $F =\ol{K} \cap \bigcap_{H \in \AAA'} H$. in particular, $F \subset \bigcap_{H \in \AAA'} H$. A point $x$ as above exists, since $F$ has nonempty interior with respect to $\bigcap_{H \in \AAA'} H$.
\end{proof}

We now make use of an abstract result for type functions of chamber complexes. The following lemma is mentioned in \cite{Mu94} as an easy consequence, but we elaborate this result.

\begin{lem}
Let $\Delta$ be a strongly connected chamber complex. If $\Delta$ is weakly numbered, it is already numbered. In particular, if $(\tau_K)_{K \in \Cham(\Delta)}$ is a weak type function, there exists a type function $\tau$ such that $\tau|_{\ul{K}} = \tau_K$ for all $K \in \Cham(\Delta)$.
\label{wnumisnum}
\end{lem}
\begin{proof}
Let $\mathcal{C} = \Cham(\Delta)$ and $(\tau_K)_{K \in \mathcal{C}}$ be a weak type function, so $\tau_K|_{\ul{K} \cap \ul{L}} = \tau_L|_{\ul{K} \cap \ul{L}}$ for all $K,L \in \mathcal{C}$.

Assume $F \in \Delta$, and let $K,L$ be chambers with $F \in \ul{K},\ul{L}$. Since $\Delta$ is strongly connected, $St(F)$ is connected, and we find $K,L \in St(F)$. Thus we find a gallery $K= K_0, K_1, \dots, K_{k-1}, K_m= L$ from $K$ to $L$ with all $K_i \in St(F)$, so in particular $F \in \ul{K_i}$ for all $0 \leq i \leq m$. Also $K_{i-1}$ and $K_i$ are adjacent for $1 \leq i \leq m$, therefore $\tau_{K_{i-1}}(F) = \tau_{K_i}(F)$ and inductively we obtain $\tau_K(F) = \tau_L(F)$.

This allows us to define $\tau(F) := \tau_K(F)$ for every simplex $F$ and every chamber $K$ containing $F$. By definition $\tau$ coincides with $\tau_K$ on all simplices $F'$ contained in $K$, in particular $\tau|_{\ul{K}}$ is a type function and thus a morphism of chamber complexes.

Finally, $\tau$ itself is a morphism, since every simplex $F$ is contained in a chamber $K$, and $\tau|_{\ul{K}}$ is a morphism.
\end{proof}

With respect to Theorem \ref{gsimpwn}, a direct consequence is the following theorem.

\begin{theo}
The complex $\SSSS$ has a type function. In particular, if we have a type function $\tau_K$ of a closed chamber $\ol{K}$, this extends uniquely to a type function of $\SSSS$.
\label{numbext}
\end{theo}

\begin{rem}
The construction of the weak type function is actually quite simple. Begin with a chamber $K$ and consider a type function $\tau$ of $\ul{\ol{K}}$. Let $L$ be adjacent to $K$ such that $F = \ol{K} \cap \ol{L}$. Then set $\tau_L|_{\ul{F}} = \tau|_{\ul{F}}$. Let $i \in I$ be the unique index such that $i \notin \tau(F)$, then $\tau$ maps the vertex not contained in $F$ to $i$, so $\tau_L$ must map the vertex $v$ in $\ul{\ol{L}}$ not contained in $F$ to $i$ as well. So as every simplex $S \neq \emptyset$ is either contained in $F$ or contains $v$, if $S$ is contained in $F$ then $\tau_L(S)$ is already defined, if it contains $v$ set $\tau_d(S) = \tau(S \cap F) \cup \{i\}$. One can check that $\tau_L$ is a morphism of chamber complexes, and furthermore $\tau_L$ is the only possible type function of $\ul{\ol{L}}$ satisfying $\tau_L|_{\ul{F}} = \tau|_{\ul{F}}$.

In this way we can inductively construct type functions for all chambers with arbitrary distance to $K$. This construction works always, however being well defined arises as a problem: Given a chamber $L$ with $d_\SSSS(K,L) = n \geq 2$, there may be two chambers $K_1,K_2$ with $d_\KKK(K,K_1) = d_\KKK(K, K_2) = n-1$ and $K_1,K_2$ adjacent to $L$. Then $L$ has induced type functions from $K_1$ as well as from $K_2$. Now Theorem \ref{numbext} yields that these two induced type functions coincide, and thus the method gives us a weak type function of $\SSSS$.
\end{rem}

Since we introduced the notion of being spherical for arrangements, we recall the respective notion for chamber complexes.

\begin{defi}
We say that a complex $\Delta$ is \textit{spherical} if it contains a pair of chambers $K,K'$ such that $\proj_P(K') \neq K$ for all panels $P$ containing $K$. The chamber $K'$ with such a property is called \textit{opposite} to $K$.
\end{defi}

Some properties of spherical complexes are the following, which can be found in \cite{Mu94}.

\begin{lem} Let $\Delta$ be a chamber complex.
\begin{enumerate}[label=\roman*)]
	\item If $\Delta$ is firm, it is spherical if and only if it has finite diameter.
	\item If $\Delta$ is meager and spherical, $K, K'$ opposite chambers. Then $\Cham(\Delta) = \sigma(K,K')$.
\end{enumerate}
\label{lem:spherprop}
\end{lem}

To complete the proof of Proposition \ref{prop-Sproperties}, we also need the following observation.

\begin{lem}
The simplicial complex $\SSSS$ is thin (resp.\ spherical) if and only if the simplicial arrangement $(\AAA,T)$ is thin (resp.\ spherical).
\end{lem}
\begin{proof}
Thin: Since $V \setminus H$ has two connected components, the complex $\SSSS$ is meager. It is thin if and only if for every chamber $C$ and every wall $H \in W^C$ there exists a chamber $C_H$ which is $H$-adjacent to $C$. In this case $\overline{C} \cap \overline{C_H} \subset T$, since $T$ is convex. Then $\overline{C} \cap \overline{C_H} \subset H$, therefore $H$ meets $T$ and is contained in $\AAA$.

Spherical: Assume $(\AAA, T)$ is spherical, then $T = V$ by definition and $\AAA$ is finite. Hence also $\KKK$ is finite and $\SSSS$ is thin, and therefore spherical by Lemma \ref{lem:spherprop}.

Let $\SSSS$ be spherical. By definition we find two opposite chambers $C$ and $C'$. As $\SSSS$ is meagre by construction, we also know $\KKK = \sigma(C,C')$. In particular $S(C,C') = \AAA$, and $\KKK$ is finite as well as $\AAA$.

Assume $C$ has no $i$-adjacent chamber for an $i \in I$ and let $P$ be the $i$-panel containing $C$, then $\proj_{P}(D) = C$ for all $D \in \KKK$, a contradiction. Thus $\SSSS$ is also thin. By our previous argument therefore $(\AAA,T)$ is thin.

Let $x \in \partial T$. Since $T$ can not be written as a finite union of hyperplanes, we can assume that $x \notin H$ for all $H \in \AAA$. Thus take a neighborhood $U$ of $x$ and consider the chambers intersecting $U$. By taking a smaller $U$ we can also assume that only a single chamber $D$ intersects $U$. Thus there exists a wall of $D$ not meeting $T$, a contradiction to $\AAA$ being thin. Hence $\partial T$ is empty and $T = V$ holds.
\end{proof}

\newpage


\providecommand{\bysame}{\leavevmode\hbox to3em{\hrulefill}\thinspace}
\providecommand{\MR}{\relax\ifhmode\unskip\space\fi MR }
\providecommand{\MRhref}[2]{%
  \href{http://www.ams.org/mathscinet-getitem?mr=#1}{#2}
}
\providecommand{\href}[2]{#2}

\end{document}